\documentclass{amsart}
\usepackage{fullpage}
\usepackage{amsthm,mathtools,tikz-cd}
\usepackage{tikz}
\usepackage[all,cmtip]{xy}
\usepackage{color}
\usepackage{bbm}
\usepackage{framed}
\usepackage{latexsym}
\usetikzlibrary{%
  matrix,%
  calc,%
  arrows%
}
\newtheorem{theorem}{Theorem}[section]
\newtheorem{lemma}[theorem]{Lemma}
\newtheorem{proposition}[theorem]{Proposition}
\newtheorem{convention}[theorem]{Convention}
\newtheorem{corollary}[theorem]{Corollary}
\newtheorem{remark}[theorem]{Remark}
\newtheorem{definition}[theorem]{Definition}
\newtheorem{example}[theorem]{Example}
\newtheorem{result}[theorem]{Result}
\newtheorem{conjecture}[theorem]{Conjecture}

\title{An Approximate Nerve Theorem}

\author{Dejan Govc}
\address[DEJAN GOVC] {Institute of Mathematics, Physics and Mechanics, Jadranska 19, 1000 Ljubljana, Slovenia}

\email[Corresponding author]{dejan.govc@imfm.si}

\author{Primoz Skraba}
\address[PRIMOZ SKRABA] { Jozef Stefan Institute, Jamova cesta 39, 1000 Ljubljana, Slovenia \&
FAMNIT, University of Primorska, Glagolja\v ska ulica 8, 6000 Koper, Slovenia
        }
\email{primoz.skraba@ijs.si}

\newcommand{\RR}{\mathbb{R}}
\newcommand{\ZZ}{\mathbb{Z}}
\newcommand{\NN}{\mathbb{N}}
\newcommand{\kk}{\mathbbm{k}}

\newcommand{\id}{\operatorname{id}}
\newcommand{\im}{\operatorname{im}}
\newcommand{\coker}{\operatorname{coker}}
\newcommand{\eps}{\varepsilon}
\newcommand{\pt}{\operatorname{pt}}
\newcommand{\projdim}{\operatorname{projdim}}
\newcommand{\CW}{{\mathrm{CW}}}

\newcommand{\gr}{{(\operatorname{Gr})}}
\newcommand{\ngr}{{(\operatorname{NGr})}}

\newcommand{\Xs}{X}
\newcommand{\Ys}{Y}
\newcommand{\XCov}{\mathcal{U}}
\newcommand{\XCv}{U}
\newcommand{\YCov}{\mathcal{V}}
\newcommand{\YCv}{V}
\newcommand{\Nrv}{\mathcal{N}}
\newcommand{\Filt}{\mathcal{F}}

\newcommand{\Dgm}[2]{\mathrm{Dgm}(#1,#2)}

\newcommand{\Vect}{\operatorname{{\bf Vect}}}
\newcommand{\Mod}{\operatorname{{\bf Mod}}}
\newcommand{\Top}{\operatorname{{\bf Top}}}
\newcommand{\SCx}{\operatorname{{\bf SCx}}}

\newcommand{\Cg}{\mathsf{C}}
\newcommand{\Zg}{\mathsf{Z}}
\newcommand{\Bg}{\mathsf{B}}
\newcommand{\Hg}{\mathsf{H}}
\newcommand{\PC}{\mathsf{C}}
\newcommand{\PH}{\mathsf{H}}

\newcommand{\Tot}{\mathrm{Tot}}

\newcommand{\itl}[1]{\overset{#1}\sim}
\newcommand{\Litl}[1]{\overset{#1}{\sim_L}}
\newcommand{\Ritl}[1]{\overset{#1}{\sim_R}}

\newcommand{\Ext}{\operatorname{Ext}}
\newcommand{\Hom}{\operatorname{Hom}}

\newcommand{\spaceD}{\Delta}
\newcommand{\minD}{Q}

\begin{document}
\begin{abstract}
 The Nerve Theorem relates the topological type of a suitably nice space with the nerve of a good cover of that space. It has many variants, such as to consider acyclic covers and numerous applications in topology including applied and computational topology. The goal of this paper is to relax the notion of a good cover to an approximately good cover, or more precisely, we introduce the notion of an $\eps$-acyclic cover. We use persistent homology to make this rigorous and prove tight bounds between the persistent homology of a space endowed with a function and the persistent homology of the nerve of an $\eps$-acyclic cover of the space. Our approximations are stated in terms of \emph{interleaving distance} between persistence modules. Using the Mayer-Vietoris spectral sequence, we prove upper bounds on the interleaving distance between the persistence module of the underlying space and the persistence module of the the nerve of the cover. To prove the best possible bound we must introduce special cases of interleavings between persistence modules called left and right interleavings. Finally, we provide examples which achieve the bound proving the lower bound and tightness of the result. 
 \end{abstract}
 \maketitle

\section{Introduction}

The Nerve Theorem is a classical result relating a sufficiently nice cover of a topological space with the nerve of that cover going back to Alexandroff~\cite{alexandroff1928allgemeinen}.
\begin{theorem}[Corollary 4G.3 \cite{Hatcher}]\label{hmtpynrv}
If $\XCov$ is an open cover of a paracompact space $\Xs$ such that every
nonempty intersection of finitely many sets in $\XCov$ is contractible, then $\Xs$ is homotopy
equivalent to the nerve $\Nrv(\Xs)$.
\end{theorem}

One more recent application is in the area of \emph{topological data analysis} (TDA) \cite{ghrist2008barcodes,Carlsson2009,weinberger2011persistent}. The goal is to obtain information about the topology of a space, often given a discrete sample of the space. There has been a large body of work proving results in different contexts, including~\cite{chazal2011scalar,bobrowski2015topology,dey2015graph} just to name a few. A common point is the use of the Nerve Theorem, either explicitly or implicitly, through constructions such as the \v Cech complex. 
A powerful tool in TDA is \emph{persistent homology}, which studies the homology of a filtration rather than a single space. Applying the homology functor yields a \emph{persistence module}. If we compute homology with field coefficients, then we can obtain a complete topological invariant called  \emph{persistence barcode} or \emph{persistence diagram}. One useful source of filtrations are sublevel (resp. superlevel) set filtrations -- given a space endowed with a real-valued continuous function, $f:\Xs\rightarrow \RR$, the sublevel sets of the function form a filtration yielding a persistence diagram denoted by $\Dgm{\Xs}{f}$.

One example of an important function is the distance to a compact set. When the compact set consists of sample points, this function relates to a notion of scale and is equivalent to the \v Cech filtration.  Recall that the \v Cech complex on a point set $P$ is the nerve of the union of balls of radius $r$. The points are usually embedded in Euclidean space, allowing the Nerve Theorem to be applied via convexity. By varying the radius $r$, we obtain the \v Cech filtration. Other filtrations which are often considered are:
superlevel sets of probability density functions~\cite{bobrowski2015topology},
sublevel set filtration of a sampled function~\cite{chazal2011scalar}, and 
the elevation function on 2-manifolds~\cite{agarwal2006extreme}.
Persistence diagrams have proven interesting because they are stable \cite{CohEdeHar2007} -- meaning a small change in the filtration bounds the change in the invariant. One way of measuring the magnitude of the change is the \emph{bottleneck distance}.
 Stability enables us to prove theorems about approximating the persistent homology of a filtration using an alternate filtration constructed from a discrete sample, i.e. that the bottleneck distance is small. 

An important technique in proving such an approximation is \emph{interleaving} \cite{chazal2012structure}, which provides an algebraic condition for approximation (Section \ref{sec:modint}). A common theme is to construct an interleaving with a good cover, providing an approximation guarantee. In some cases, such as for distance filtration, an interleaving with a good cover can often be shown directly. In more general settings, it can often be more difficult to directly prove an interleaving. Our goal in this paper is to prove an approximation bound using the stability of persistent homology to relax the need for a good cover. Importantly, we only make assumptions on the \emph{local properties} of the space and function, which make it useful in a variety of applications. 

As we deal with persistent homology, we concentrate on a homological version of Theorem \ref{hmtpynrv}.  
\begin{theorem}[Theorem 4.4 \cite{brown2012cohomology}]\label{thm:acyclic}
Suppose $\Xs$ is the union of subcomplexes $\XCv_i$ such that every
non-empty intersection $\XCv_{i_0}\cap \cdots  \cap \XCv_{i_p}$ for $p \geq 0$ is acyclic. Then 
$\Hg_*(\Xs) \cong \Hg_*(\Nrv(\XCov))$
where $\Nrv(\XCov)$ is the nerve of the cover.
\end{theorem} 

The main result of this paper is to provide an \emph{approximate} version of the above theorem in the context of persistent homology. Given a space and function, we first define the notion of an $\eps$-acyclic cover.  Note that we do not restrict ourselves to induced functions on a fixed cover, but consider a covering by filtrations. This notion is less intuitive but is applicable in a wider range of settings. We follow the formalization of covers by filtrations by Sheehy \cite{sheehy2012multicover}. Informally, our main result is:
\begin{result}
Given a space $\Xs$ endowed with a function $f$ and a (filtered) cover $\XCov$, if every non-empty finite intersection of cover elements is $\eps$-acyclic, then there exists a function on the nerve $g: \Nrv(\XCov)\rightarrow \RR$, such that the bottleneck distance $d_B(\cdot)$ is bounded by
\[
d_B(\Dgm{\Xs}{f}, \Dgm{\Nrv(\XCov)}{g})\leq 2(Q+1)\eps,
\]
where 
\[
Q = \min\{\mathrm{dim} (\Xs), \mathrm{dim} (\Nrv(\XCov)\}.
\]
\end{result}
The construction of the function on the nerve is given explicitly and agrees with what is currently done in practice when computing persistent homology.  

In the paper, we do not use persistence diagrams and bottleneck distance, but find it more convenient to work directly with the corresponding persistence modules and interleavings. Therefore, we do not explicitly define bottleneck distance or diagrams as they are not required for the statement of our results, but do allude to them to give intuition for readers who are familiar with persistence. In cases where the diagrams are well defined, bottleneck distance type of results follow automatically. 

We prove this result by using the Mayer-Vietoris spectral sequence to glue together the $\eps$-acyclic pieces into the global persistent homology. To obtain a tight bound, we introduce the notion of \emph{left} and \emph{right-interleavings} (Section~\ref{sec:interleaving}), which have additional structure. This refinement of interleavings captures similar phenomena as the results of \cite{bauer2014induced}, but works at the level of modules rather than barcodes. Hence, it does not require modules to be decomposable and so we believe these notions are of independent interest.

We prove the result in two steps: first we show how the approximation bound evolves through the computation of the spectral sequence (Section ~\ref{sec:approx}), then resolve the extension problem to relate the result of the spectral sequence with the persistent homology of the underlying space (Section~\ref{sec:linking}). While we have tried to make this paper self contained, we do assume some familiarity with spectral sequences, but we try and provide intuition and references whenever possible.

\section{Preliminaries}
\label{sec:prelim}
We assume the reader is familiar with persistent homology. We refer the reader to \cite{EdeHar2010} and \cite{zomorodian2005computing} for complete introductions. The relevant preliminaries are given below -- as much as possible we have tried to avoid technical complications but we try to point out where generalizations are possible.

\subsection{Topological preliminaries}
To minimize technical complications,  we work primarily with $\ZZ$-filtered simplicial complexes (see Definition~\ref{def:Zfiltered} below), denoted $\Xs$ and $\ZZ$-filtered covers $\XCov$ of such complexes by subcomplexes (where each subcomplex itself has a specified filtration). However, our proofs work directly on the algebraic level and hence should be extendable to much more general settings than those presented here without changing the bounds. 
Note that already the results for $\ZZ$-filtered simplicial complexes are widely applicable. It is known, for instance, that each smooth manifold or, more generally, Whitney-stratified space can be triangulated. Hence, such a space $\Ys$ equipped with a sublevel set filtration induced by some function $f:\Ys\to\RR$ can be approximated arbitrarily well by a piecewise-linear (PL) function on a simplicial complex and therefore by an $\RR$-filtered simplicial complex.

\begin{definition}\label{def:Zfiltered}
Let $J\subseteq\RR$. A {\em $J$-filtered simplicial complex} is a pair $(\Xs,\Filt)$, where $\Xs$ is an abstract simplicial complex and $\Filt=(\Xs^j)_{j\in J}$ is a family of subcomplexes such that $j_1\leq j_2$ implies $\Xs^{j_1}\subseteq \Xs^{j_2}$, $\Xs^{-\infty}:=\bigcap_{j\in J}\Xs^j=\emptyset$ and $\Xs^{\infty}:=\bigcup_{j\in J}\Xs^j=\Xs$.

A {\em $J$-filtered cover\footnote{All covers are assumed to be indexed.} by subcomplexes} of a $J$-filtered simplicial complex $(\Xs,\Filt)$ is an indexed family $\XCov=(\XCv_i,\Filt_i)_{i\in\Lambda}$, where each $\XCv_i$ is a subcomplex of $\Xs$ and $\Filt_i$ is a filtration of this subcomplex, such that the filtrations $\Filt$ and $\Filt_i$ satisfy a compatibility requirement, namely that $\Xs^j=\bigcup_{i\in\Lambda}\XCv_i^j$ holds for each $j\in J$. Note that whenever $I\subseteq\Lambda$, the intersection $\XCv_I:=\bigcap_{i\in I}\XCv_i$ has a natural filtration $\Filt_I$ given by $\XCv_I^j:=\bigcap_{i\in I}\XCv_i^j$.
\end{definition}

Note that the requirements on $\Xs^{-\infty}$ and $\Xs^{\infty}$ are sometimes dropped. If $J\subseteq\RR$ is a discrete subset, for instance $J=\ZZ$, the filtration $\Filt$ may also be given as a function $f:\Xs\to\ZZ$ whose sublevel sets are $f^{-1}(-\infty,j]=\Xs^j$. For this reason, a $\ZZ$-filtered simplicial complex is sometimes written as $(\Xs,f)$. Since the filtration is regarded as part of the structure, we often suppress it from notation and simply write $\Xs$.

When the filtrations are given as functions, the compatibility requirement in the definition of the filtered cover $\XCov=(\XCv_i,f_i)_{i\in\Lambda}$ of the filtered simplicial complex $(\Xs,f)$ can be stated\footnote{To make sense of the minimum, we may consider $f_i$ to be extended to the whole $\Xs$ by defining it to be $\infty$ outside of $\XCv_i$.} as $f=\min_{i\in\Lambda}f_i$.

\begin{remark}\label{inducedcov}
This definition of $J$-filtered cover is the one given by Sheehy \cite{sheehy2012multicover}, which allows for the extension of Theorem \ref{thm:acyclic} to the persistent setting via the Persistent Nerve Lemma of Chazal and Oudot \cite{chazal2008towards}.

There is also a natural way to assign a filtered cover to an unfiltered cover of a filtered complex. Namely, if $(\Xs,\Filt)$ is a $J$-filtered simplicial complex and $\overline\XCov=(\XCv_i)_{i\in\Lambda}$ is a cover of the underlying complex $\Xs$ by subcomplexes, $\overline\XCov$ can naturally be given the structure of a $J$-filtered cover $\XCov=(\XCv_i,\Filt_i)_{i\in\Lambda}$ by defining $\XCv_i^j=\XCv_i\cap\Xs^j$. We call $\XCov$ the {\em induced $J$-filtered cover of $(\Xs,\Filt)$} associated to $\overline\XCov$. In this case, if the filtrations are given by functions $f:\Xs\to\ZZ$ and $f_i:\XCv_i\to\ZZ$, the functions $f_i$ are simply restrictions $f_i=f|_{\XCv_i}$.
\end{remark}

Our results also make sense in the setting of triangulable spaces, which we now recall.

\begin{definition}
A topological space $\Ys$ is said to be {\em triangulable} if there exists a simplicial complex $\Xs$ and a homeomorphism $h:|\Xs|\to\Ys$, where $|\Xs|$ denotes the carrier of $\Xs$. The pair $(\Xs,h)$ is said to be a {\em triangulation} of $\Ys$.
\end{definition}

In the persistent setting, we also consider filtered triangulable spaces. To do this, start with a space $\Ys$ and a continuous function $f:\Ys\to\RR$. The pair $(\Ys,f)$ is then regarded to be an $\RR$-filtered topological space. The filtration is defined by $\Ys^j=f^{-1}(-\infty,j]$ and is known as {\em the sublevel set filtration of $\Ys$ induced by $f$}.

We sometimes need to replace the function $f:\Ys\to\RR$ by a piecewise linear approximation.

\begin{definition}
Suppose $(\Xs,h)$ is a triangulation of $\Ys$ and $f:\Ys\to\RR$ a continuous function. The {\em piecewise linear approximation of $f$ associated to $(\Xs,h)$} is the function $\hat f:|\Xs|\to\RR$ defined on the vertices of $\Xs$ by $\hat f(v)=f(h(v))$ and extended affinely over the simplices.
\end{definition}

\begin{definition}
Suppose $\Xs$ is a simplicial complex and $\hat f:|\Xs|\to\RR$ a piecewise linear function (w.r.t. the triangulation). Then $\Xs$ may be given the structure of an $\RR$-filtered simplicial complex $(\Xs,\hat f)$ by defining $\Xs^j$ to consist of all simplices contained in $\hat f^{-1}(-\infty,j]$. We call this filtration {\em the lower star filtration of $\hat f$}.
\end{definition}

Note that lower star filtrations are usually considered only for finite simplicial complexes and the function values on the vertices are assumed to be distinct. (See, for instance, \cite{EdeHar2010}.)

It is a standard fact (for finite simplicial complexes, this is explained in \cite{EdeHar2010}) that the sublevel set filtration of $|\Xs|$ and the lower star filtration of $\Xs$ induced by the same function $\hat f$ are related by the fact that $|\Xs^j|$ is a deformation retract of $|\Xs|^j$. Of more interest to us, however, is comparing the persistence modules of these two filtrations.

Finally, we recall the standard construction for the nerve given a cover is:
\begin{definition}\label{def:nerve}
Given a cover $(\XCv_i)_{i\in\Lambda}$ of $\Xs$, the nerve $\Nrv$ is the set of finite subsets of $\Lambda$ defined as follows:
a finite set $I\subseteq\Lambda$ belongs to $\Nrv$ if and only if the intersection of the $\XCv_i$ whose indices are in $I$, is non-empty, or equivalently
\[
\XCv_I=\bigcap_{i\in I}\XCv_i\neq \emptyset.
\]
If $I$ belongs to $\Nrv$, then so do all of its subsets making $\Nrv$ an abstract simplicial complex.
\end{definition}

\subsection{Modules and interleavings}\label{sec:modint}

Let $\kk$ be a field. Both the graded and the non-graded ring of polynomials with coefficients in $\kk$ are commonly denoted by $\kk[t]$ in the literature. We mostly work with the former. For this reason, we reserve the notation $\kk[t]=\kk[t]_\gr$ for the graded version and the non-graded version is always explicitly denoted as such by $\kk[t]_\ngr$.

Here $\kk[t]$ is graded by degree, namely $\kk[t]=\bigoplus_{i\in\NN_0}\kk[t]_i$, where $\kk[t]_i=\kk\cdot t^i$ consists of the homogeneous polynomials of degree $i$, i.e. scalar multiples of $t^i$. This decomposition is regarded as part of the structure of $\kk[t]$ and has to be taken into account when defining $\kk[t]$-modules and their morphisms, whereas $\kk[t]_{\ngr}$ is simply a ring without any additional structure, so $\kk[t]_{\ngr}$-modules and their morphisms are not required to respect any such grading.

\begin{definition}\cite[p. 42]{eisenbud2013commutative}
A {\em $\kk[t]$-module} is a $\kk[t]_\ngr$-module $M$ together with a decomposition (also called grading) into abelian subgroups $M=\bigoplus_{j\in\ZZ}M^j$ such that $\kk[t]_i\cdot M^j\subseteq M^{i+j}$ holds for all $i\in\NN_0$ and $j\in\ZZ$. Let $\eps\in\NN_0$. An {\em $\eps$-morphism of $\kk[t]$-modules} $M$ and $N$ is a morphism $f:M\to N$ of the underlying $\kk[t]_\ngr$-modules such that for all $j\in\ZZ$ we have $f(M^j)\subseteq N^{j+\eps}$. A $0$-morphism is also called {\em a morphism}.
\end{definition}

\begin{example}
There is a distinguished $\eps$-morphism $\id_\eps:M\to M$ given by $\id_\eps(m)=t^\eps m$.
\end{example}

Since $\kk[t]_0=\kk$, any such $M$ is also a $\ZZ$-graded $\kk$-module. Consequently, some authors \cite{eisenbud2013commutative} call this a {\em graded $\kk[t]$-module}. For us, ``graded $\kk[t]$-module'' means something else (see Definition~\ref{def:graded-module}).

\begin{definition}
An {\em $\eps$-interleaving of $\kk[t]$-modules} $M$ and $N$ is a pair $(\phi,\psi)$ of $\eps$-morphisms $\phi:M\to N$ and $\psi:N\to M$ such that $\phi\psi=\id_{2\eps}$ and $\psi\phi=\id_{2\eps}$. A $0$-interleaving is the same as an isomorphism. If there is an $\eps$-interleaving between $M$ and $N$, we say that $M$ and $N$ are {\em $\eps$-interleaved} and write $M\itl{\eps}N$.
\end{definition}

\begin{remark}
We also work with interleavings of graded modules and chain complexes. These are defined by components and for the latter we additionally require that the interleaving maps are chain maps, i.e. that they commute with the differentials.
\end{remark}

The notion of $\eps$-interleaving defines an extended\footnote{This means that we allow it to take the value $\infty$.} metric between isomorphism classes of $\kk[t]$-modules, i.e. it satisfies the following basic properties.

\begin{proposition}\label{distance}
Suppose $M, N$ and $P$ are $\kk[t]$-modules. Then the following properties hold.
\begin{enumerate}
\item Positive definiteness: $M\itl{0}N$ holds if and only if $M\cong N$.
\item Symmetry: $M\itl{\eps}N$ implies $N\itl{\eps}M$.
\item Triangle inequality: $M\itl{\eps_1}N$ and $N\itl{\eps_2}P$ imply $M\itl{\eps_1+\eps_2} P$.
\end{enumerate}
\end{proposition}

\begin{proof}
The first two properties are immediate. To show the third, let $(\phi,\psi)$ be an $\eps_1$-interleaving of $M$ and $N$ and $(\eta,\theta)$ an $\eps_2$-interleaving of $N$ and $P$. Then, $(\eta\phi,\theta\psi)$ is an $(\eps_1+\eps_2)$-interleaving of $M$ and $P$.
\end{proof}

\begin{definition}
The {\em interleaving distance} between $\kk[t]$-modules $M$ and $N$ is defined by the formula
\[
d_I(M,N)=\min\{\eps\in\NN_0\mid M\itl{\eps}N\}.
\]
\end{definition}

Therefore, the notion of interleaving provides a means to quantify how close two modules are to each other. Modules $\eps$-interleaved with $0$ are of particular importance, as they may be regarded as small, and are therefore useful as a model of experimental error. Alternatively, they are characterized as follows.

\begin{proposition}
A $\kk[t]$-module $M$ is $\eps$-interleaved with $0$ if and only if $t^{2\eps}M=0$.
\end{proposition}

\begin{proof}
If $M\itl{\eps}0$, let $(\phi,\psi)$ be the interleaving. This means that $\id_{2\eps}=\psi\phi=0$. Conversely, if $t^{2\eps}M=0$, the interleaving is given by $(\phi,\psi)=(0,0)$.
\end{proof}

This immediately implies that subquotients of small modules are small.

\begin{corollary}\label{subquotient}
Let $M$ be a $\kk[t]$-module and $P$ its subquotient. Then $M\itl{\eps}0$ implies $P\itl{\eps}0$.
\end{corollary}

\begin{proof}
Let $P=N/_\sim$ for some $N\leq M$. Since $t^{2\eps}M=0$, we have $t^{2\eps}N=0$ and therefore $t^{2\eps}P=0$.
\end{proof}

In the context of persistence modules it is useful to define the notions of interleavings categorically.
Let $\Vect$ be the category of vector spaces over $\kk$ and let $I\subseteq\RR$ be closed under addition. Being a poset, $I$ may be viewed as a category in the usual way. For each $\eps\in I$ and $\eps\geq 0$, there is a functor $T_\eps:(I,\leq)\to(I,\leq)$ given by $T_\eps(a)=a+\eps$ and a natural transformation $\eta_\eps:\id\Rightarrow T_\eps$ given by $\eta_\eps(a):a\to a+\eps$. These observations are due to Bubenik and Scott \cite{bubenik2014categorification}. This leads to the following definition.

\begin{definition}
A {\em persistence module} is a functor $F:(I,\leq)\to\Vect$. For $\eps\geq 0$, an {\em $\eps$-morphism} $\phi:F\overset{\eps}\to G$ is a natural transformation $\phi:F\Rightarrow G\circ T_\eps$. A {\em morphism} is a $0$-morphism. An $\eps_1$-morphism and an $\eps_2$-morphism can be composed in the natural way to yield a $(\eps_1+\eps_2)$-morphism. An {\em $\eps$-interleaving} is a pair $(\phi,\psi)$ of $\eps$-morphisms $\phi:F\overset{\eps}\to G$ and $\psi:F\overset{\eps}\to G$ such that $\psi\phi=F\eta_{2\eps}$ and $\phi\psi=G\eta_{2\eps}$. We say $F$ and $G$ are {\em $\eps$-interleaved}, $F\itl{\eps}G$. 
\end{definition}

We denote the corresponding functor category by $\Vect^{(I,\leq)}$. 
The notion of {\em interleaving distance} also makes sense in the setting of persistence modules and is defined by the analogous formula
\[
d_I(F,G)=\inf\{\eps\in I\cap[0,\infty)\mid F\itl{\eps}G\}.
\]
Note however that the infimum is not necessarily attained in this case (see \cite{chazal2014observable}).

A standard fact about persistence modules over $I=\ZZ$ is that they correspond in a natural way to $\kk[t]$-modules. Let $\Mod_{\kk[t]}$  denote the category of modules over $\kk[t]$ (in a graded sense). Then, the following holds.

\begin{theorem}
The categories $\Vect^{(\ZZ,\leq)}$ and $\Mod_{\kk[t]}$ are isomorphic.
\end{theorem}

\begin{proof}
Inverse functors $\Phi:\Vect^{(\ZZ,\leq)}\to\Mod_{\kk[t]}$ and $\Psi:\Mod_{\kk[t]}\to\Vect^{(\ZZ,\leq)}$ can be defined explicitly. On objects, these are defined as $\Phi(F)=\bigoplus_{j\in\ZZ}F(j)$ and $\Psi(M)(j)=M^j$. On morphisms, we have $\Phi(\eta)=(\eta_j)_{j\in\ZZ}$ and $\Psi(f)_j=f_j$, where $f_j:M^j\to N^j$ is the restriction of $f$ to the $j$-th step of the filtration.
\end{proof}

Note that this also holds for $I=\eps\ZZ$, $\eps>0$. The correspondence between persistence modules over $I=\ZZ$ and $\kk[t]$-modules was first noted in \cite{zomorodian2005computing}.  We use this extensively in this paper. There is a similar correspondence between persistence modules over $I=\RR$ and modules over the monoid algebra over $[0,\infty)$, as noted by Lesnick \cite{lesnick2015theory}.
However, for our application, as we shall see, this is unnecessary.

In particular, we can show that each persistence module over $\RR$ can be approximated by a persistence module over $\eps\ZZ$ up to $\eps$. To make sense of this, first observe that there is a natural inclusion functor $i_\eps:(\eps\ZZ,\leq)\to(\RR,\leq)$. This functor has a left inverse $p_\eps:(\RR,\leq)\to(\eps\ZZ,\leq)$ given by $p_\eps(a)=\lfloor\frac{a}{\eps}\rfloor\eps$. Note that this left inverse is not unique. In a sense, however, it is the most natural choice in our situation.

These two functors give rise to the (natural) restriction functor $I_\eps:\Vect^{(\RR,\leq)}\to\Vect^{(\eps\ZZ,\leq)}$ given by $I_\eps(F)=Fi_\eps$ and an extension functor $P_\eps:\Vect^{(\eps\ZZ,\leq)}\to\Vect^{(\RR,\leq)}$ given by $P_\eps(F)=Fp_\eps$. Under our choice of $p_\eps$, when defined, the persistence diagrams of $F:(\eps\ZZ,\leq)\to\Vect$ and $P_\eps(F):(\RR,\leq)\to\Vect$ agree as multisets (except perhaps on the diagonal, depending on the convention used).

The functors $I_\eps$ and $P_\eps$ have various useful properties. Since $p_\eps i_\eps=\id$, we have $I_\eps P_\eps=\id$. The composition $P_\eps I_\eps$ is, in a certain sense, also not far from the identity. Furthermore, $P_\eps$ is an isometric embedding, and $I_\eps$ is an almost isometry.

\begin{proposition}\label{rightinverse}
Let $F:(\RR,\leq)\to\Vect$ be a persistence module. Then $F$ and $P_\eps I_\eps(F)$ are $\eps$-interleaved.
\end{proposition}

\begin{proof}
An $\eps$-interleaving $(\phi,\psi)$ is given by $\phi_x:F(x)\to F(p_\eps(x)+\eps)$ and $\psi_x:F(p_\eps(x))\to F(x+\eps)$, given by the shifting morphisms $\phi_x=\id_{p_\eps(x)+\eps-x}$ and $\psi_x=\id_{x+\eps-p_\eps(x)}$.
\end{proof}

\begin{proposition}\label{isometry}
The functor $P_\eps:\Vect^{(\eps\ZZ,\leq)}\to\Vect^{(\RR,\leq)}$ is an isometric embedding.
\end{proposition}

\begin{proof}
Suppose $F,G\in\Vect^{(\eps\ZZ,\leq)}$ and suppose $F\itl{\eta}G$ and let $(\phi,\psi)$ be the relevant interleaving. Then, $(\phi p_\eps,\psi p_\eps)$ is an $\eta$-interleaving of $P_\eps(F)$ and $P_\eps(G)$. This implies that
\[
d_I(F,G)=\min\{\eta\in\eps\NN_0\mid F\itl{\eta}G\}\geq\inf\{\eta\in[0,\infty)\mid P_\eps(F)\itl{\eta}P_\eps(G)\}=d_I(P_\eps(F),P_\eps(G)).
\]
To prove the converse inequality, suppose $P_\eps(F)$ and $P_\eps(G)$ are $\eta$-interleaved and let $(\phi,\psi)$ be the interleaving. We claim that this implies $F$ and $G$ are $p_\eps(\eta)$-interleaved. To define an interleaving, note that for $k\in \ZZ$,  $\phi_{k\eps}:F(k\eps)\to G(p_\eps(k\eps+\eta))=G(k\eps+p_\eps(\eta))$ and $\psi_{k\eps}:G(k\eps)\to F(p_\eps(k\eps+\eta))=F(k\eps+p_\eps(\eta))$, so the maps $\tilde\phi_{k\eps}=\phi_{k\eps}$ and $\tilde\psi_{k\eps}=\psi_{k\eps}$ are components of a $p_\eps(\eta)$-interleaving $(\tilde\phi,\tilde\psi)$, showing that $d_I(F,G)\leq d_I(P_\eps(F),P_\eps(G))$.
\end{proof}

\begin{proposition}\label{discreteitl}
Given persistence modules $F,G:(\RR,\leq)\to\Vect$, we have 
\[
d_I(F,G)-2\eps\leq d_I(I_\eps(F),I_\eps(G))\leq d_I(F,G)+\eps.
\]
\end{proposition}

\begin{proof}
Let $A=\{\eta\in\eps\NN_0\mid I_\eps(F)\itl{\eta}I_\eps(G)\}$ and $B=\{\eta\in [0,\infty)\mid F\itl{\eta}G\}$. Note that if two modules are $\eta$-interleaved, they are also $\theta$-interleaved for all $\theta\geq\eta$, so these sets are upward closed in $\eps\NN_0$ and $[0,\infty)$, respectively. By definition, we have
\[
d_I(I_\eps(F),I_\eps(G))=\min A\qquad\text{and}\qquad d_I(F,G)=\inf B.
\]
Suppose $\eta\in B\cap \eps\NN_0$ and let $(\phi,\psi)$ be the relevant $\eta$-interleaving. Then $(\phi i_\eps,\psi i_\eps)$ is an $\eta$-interleaving of $I_\eps(F)$ and $I_\eps(G)$. Therefore, $B\cap\eps\NN_0\subseteq A$. Since $B$ is upward closed, this immediately implies $\inf B\geq\min A-\eps$ and therefore
\[
d_I(I_\eps(F),I_\eps(G))\leq d_I(F,G)+\eps.
\]
The other inequality follows from Proposition \ref{isometry} and Proposition \ref{rightinverse}:
\begin{multline*}
d_I(F,G)\leq d_I(F,P_\eps(I_\eps(F)))+d_I(P_\eps(I_\eps(F)),P_\eps(I_\eps(G)))+d_I(P_\eps(I_\eps(G)),G)\\\leq d_I(P_\eps(I_\eps(F)),P_\eps(I_\eps(G)))+2\eps=d_I(I_\eps(F),I_\eps(G))+2\eps.
\end{multline*}
\end{proof}

These observations allow us to compare persistence modules over $\eps\ZZ$ and persistence modules over $\RR$. Namely, since $P_\eps$ is an isometric embedding, $\eps\ZZ$-persistence modules can be understood as a special case of $\RR$-persistence modules, namely those satisfying the property $F(a\to b)=\id$ for any pair of points $a\leq b$ lying the same interval $[k\eps,(k+1)\eps)$. Therefore, we regard persistence modules $F:(\eps\ZZ,\leq)\to\Vect$ and $G:(\RR,\leq)\to\Vect$ as $\eps$-close if $P_\eps(F)$ and $G$ are $\eps$-interleaved. With this understanding, we may state:

\begin{corollary}
For any $\eps>0$, any continuous-valued persistence module $F:(\RR,\leq)\to\Vect$ can be $\eps$-approximated by a $\kk[t]$-module $F_\eps:(\eps\ZZ,\leq)\to\Vect$, namely $F_\eps=I_\eps(F)$. 
\end{corollary}

We concern ourselves with strictly positive $\eps$. The connection between discrete and continuous parameter persistence was first exploited in the first algebraic persistence stability result \cite{ChaCohGli2009} and has been studied in \cite{vejdemo2012interleaved}. The related notion of \emph{observable structure} was further introduced in \cite{chazal2014observable}. In principle, this discretization is technically unnecessary but desirable in algorithmic applications (see Discussion). 

As mentioned at the end of the preceding section, we would like to compare the persistence modules of a sublevel set filtration and a lower star filtration associated to the same piecewise linear function $\hat f$. The functorial approach is fruitful here, as the two filtrations may also be regarded as functors $S_{\hat f}:(\RR,\leq)\to(\Top,\subseteq)$ and $L_{\hat f}:(\RR,\leq)\to(\SCx,\subseteq)$, respectively.

Let $\Hg_n$ denote the $n$-th simplicial homology functor and $\Hg_n^s$ the $n$-th singular homology functor. Note that these are related by $\Hg_n\cong\Hg_n^sG$, where $G:(\SCx,\subseteq)\to(\Top,\subseteq)$ is the geometric realization functor.

\begin{proposition}\label{filtrations}
Suppose $\Xs$ is a simplicial complex and $\hat f:|\Xs|\to\RR$ a piecewise linear function (w.r.t. the triangulation). Then the persistence modules $\Hg^s_nS_{\hat f}:(\RR,\leq)\to\Vect$ and $\Hg_nL_{\hat f}:(\RR,\leq)\to\Vect$ are isomorphic.
\end{proposition}

\begin{proof}
There is a natural transformation $\eta:G\circ L_{\hat f}\Rightarrow S_{\hat f}$ given component-wise by the inclusions $|\Xs^j|\to|\Xs|^j$. However, since $|\Xs^j|$ and $|\Xs|^j$ are homotopy equivalent, the components of the natural transformation $\Hg_n^s\eta:\Hg_nL_{\hat f}\to\Hg_n^sS_{\hat f}$ are isomorphisms, therefore it is a natural isomorphism.
\end{proof}

Another important fact about sublevel set filtrations is that the persistent homologies associated to a pair of $\eps$-close functions on the same space are $\eps$-interleaved.  We recall a classical result. 

\begin{proposition}\label{functionalcase}
Suppose $\Ys$ is a topological space and $f,g:\Ys\to\RR$ are functions satisfying $\|f-g\|_\infty\leq\eps$. Then the persistence modules $\PH_*(\Ys,f)$ and $\PH_*(\Ys,g)$ are $\eps$-interleaved.
\end{proposition}

\begin{proof}
For each $x\in\RR$ there are inclusions $f^{-1}(-\infty,x]\to g^{-1}(-\infty,x+\eps]$ and $g^{-1}(-\infty,x]\to f^{-1}(-\infty,x+\eps]$. Upon taking their homology, we obtain the desired $\eps$-interleaving.
\end{proof}

\begin{remark}\label{functionalcase2}
This also holds for lower star filtrations, with completely analogous proof.
\end{remark}

As we have seen, continuous persistence modules can be approximated by discrete ones. For this reason, we mostly work with $\kk[t]$-modules in the remainder of the paper. To avoid notational clutter, we also adopt the following convention.

\begin{convention}
Both ordinary and persistent homology are denoted by the same symbol $\PH_*$. In case the filtration is explicitly mentioned, as in $\PH_*(\Xs,\Filt)$ or $\PH_*(\Xs,f)$, the meaning is unambiguous. However, when suppressing the filtration, $\PH_*(\Xs)$ could in principle mean either the persistent homology of the filtered simplicial complex $(\Xs,\Filt)$ or the ordinary homology of its underlying space $\Xs$. Whenever $\Xs$ has the structure of a filtered simplicial complex, $\PH_*(\Xs)$ will always mean persistent homology and $\PH_*(\Xs^j)$, with the filtration step explicitly specified (possibly $j=\infty$), will always mean ordinary homology.
\end{convention}

\subsection{Spectral Sequences}\label{sec:specsec}

In this section, we introduce the concept of a \emph{spectral sequence} and examine its various basic properties. Then, we focus our attention on the Mayer-Vietoris spectral sequence which is the one most suitable for our needs.
Many spectral sequences arise from double complexes. A description of these spectral sequences can be found in \cite[Chapter 10]{rotman2008introduction} and \cite{mccleary2001user}. Versions of the Mayer-Vietoris spectral sequence can also be found in \cite{bott2013differential} and \cite{brown2012cohomology} among numerous others. 


\begin{definition}\label{def:graded-module}
{\em A graded $\kk[t]$-module} is a $\ZZ$-indexed family $M=M_*=(M_p)_{p\in\ZZ}$ of $\kk[t]$-modules.
\end{definition}

\begin{definition}
{\em A (chain) complex of $\kk[t]$-modules} is a pair $(\Cg,\partial)$ where $\Cg$ is a graded $\kk[t]$-module and $\partial=(\partial_p)_{p\in\ZZ}$ is a family of morphisms $\partial_p:\Cg_p\to\Cg_{p-1}$ of $\kk[t]$-modules such that $\partial_{p-1}\partial_p=0$ for each $p\in\ZZ$.
\end{definition}

It is often convenient to view a graded $\kk[t]$-module as a genuine $\kk[t]$-module by identifying it with the direct sum of its components $M\equiv\bigoplus_{p\in\ZZ}M_p$. The decomposition is regarded as part of the structure. Similarly, we often view a chain complex as a {\em differential graded module}, i.e. the $\kk[t]$-module $\Cg\equiv\bigoplus_{p\in\ZZ}\Cg_p$ equipped with a $\kk[t]$-module homomorphism $\partial$ such that $\partial(\Cg_p)\subseteq \Cg_{p-1}$ and $\partial\circ\partial=0$.

\begin{definition}
{\em A bigraded $\kk[t]$-module} is a $\ZZ\times\ZZ$-indexed family $M=M_{*,*}=(M_{p,q})_{p,q\in\ZZ}$ of $\kk[t]$-modules.
\end{definition}

\begin{definition}
{\em A double complex (bicomplex) of $\kk[t]$-modules} is a triple $(M,\partial^0,\partial^1)$ where $M$ is a bigraded $\kk[t]$-module and $\partial^0=(\partial^0_{p,q})_{p,q\in\ZZ}$ and $\partial^1=(\partial^1_{p,q})_{p,q\in\ZZ}$ are two families of morphisms $\partial^0_{p,q}:M_{p,q}\to M_{p,q-1}$ and $\partial^1_{p,q}:M_{p,q}\to M_{p-1,q}$ such that $\partial^0_{p,q-1}\partial^0_{p,q}=0,\partial^1_{p-1,q}\partial^1_{p,q}=0$ and $\partial^1_{p,q-1}\partial^0_{p,q}+\partial^0_{p-1,q}\partial^1_{p,q}=0$ for $p,q\in\ZZ$.
\end{definition}

Note that the notions of $\eps$-morphisms and interleavings make sense for bigraded modules and therefore for spectral sequences. We may define them by components and, in the case of double complexes, additionally assume that they commute with both differentials.

As with graded modules and complexes, we often view bigraded $\kk[t]$-modules as genuine $\kk[t]$-modules with additional structure, namely via the identification $M=\bigoplus_{p,q\in\ZZ}M_{p,q}$. We can also view $M$ as a graded $\kk[t]$-module in (at least) two ways, namely by summing over all $p$ or by summing over all $q$. Using this view, a double complex can be seen as a bigraded module $M$ that is a differential module with respect to $\partial^0$ as well as with respect to $\partial^1$, and the two structures are related by the equation $\partial^0\partial^1+\partial^1\partial^0=0$.

The relevance of this anticommutativity property is that combining the two differentials by summing them also yields a differential $\partial^0+\partial^1$. In fact, we may equivalently work with {\em commutative double complexes} $(M,\partial^0,\partial^1)$ with the only difference that $\partial^0$ and $\partial^1$ commute instead of anticommute, i.e. $\partial^0\partial^1=\partial^1\partial^0$. Note that such $M$ becomes an anticommutative double complex upon replacing $\partial^0$ by $(-1)^p\partial^0$. The advantage of the anticommutative case is that we do not have to keep track of signs in the combined differential.
\begin{center}
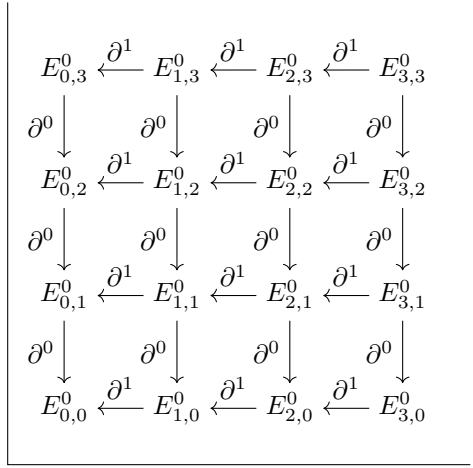
\begin{figure}[h]
\begin{tikzpicture}[scale=1.5]
  \foreach \x in {0,...,3}
    \foreach \y in {0,...,3} 
	{
		\node (E\x\y) at (\x,\y) {$E^0_{\x,\y}$};
	}
   \foreach \x/\xx in {0/1,1/2,2/3}
    \foreach \y in {0,...,3} 
	{
		\draw[->] (E\xx\y) --node[above]{$\partial^1$} (E\x\y);
	}
   \foreach \x/\xx in {0/1,1/2,2/3}
    \foreach \y in {0,...,3} 
	{
		\draw[->] (E\y\xx) --node[left]{$\partial^0$} (E\y\x);
	}

	\draw (-0.5,3.6) -- (-0.5,-0.5) -- (3.6,-0.5);
\end{tikzpicture}
\caption{A double complex comes equipped with two differentials $\partial^0$ and $\partial^1$. Considering the antidiagonals, we also obtain a chain complex, called the total complex with $\partial^0\partial^1+ \partial^1\partial^0=0$ by anticommutativity.}
\end{figure}
\end{center}


To each double complex, one may associate a total complex by summing over the antidiagonals  and combining the two boundary operators into a total boundary operator. Note that  the $n$-th antidiagonal is the direct sum of all entries in the double complex such that $p+q=n$. This leads to the following definition.

\begin{definition}
Let $M$ be a double complex. The {\em total complex} $(\Tot(M),D)$ associated to $M$ is the chain complex defined by $\Tot_n(M)=\bigoplus_{p+q=n}M_{p,q}$ and $D=\partial^0+\partial^1$. 
\end{definition}

Spectral sequences are a tool that allows us to compute the homology of this total complex. This is a very common situation in practice. Suppose we are given a chain complex $(\Cg,\partial)$ whose homology we would like to compute. It is often possible to find a natural filtration of such a complex. By taking successive quotients, one then obtains a double complex $M$, whose total complex is isomorphic to the original chain complex. In particular, their homologies agree:
\[
\PH_*(\Tot(M),D) \cong \PH_*(\Cg,\partial).
\]
The homology of such a complex $(\Cg,\partial)$ can therefore be computed systematically using the associated spectral sequence. In fact, this is precisely what happens in our case. The associated \emph{spectral sequence} consists of \emph{pages}, where each page $E^r, r=0,1,\ldots$ is a differential bigraded module, computed successively by taking the homology with respect to the differential on the previous page. On the $r$-th page, the differential is given by
\[
d^r: E^r_{p,q} \rightarrow E^r_{p-r,q+r-1}.
\]
It may happen that there is a $R$ such that for $r>R$ all differentials beginning or ending at $E^r_{p,q}$ are zero maps. In this case the $(p,q)$-th component stabilizes in the sense that all these modules $E^r_{p,q}$ are isomorphic. If such a $R$ exists for each pair $(p,q)$, the spectral sequence is said to \emph{converge} and the stabilized modules $E^r_{p,q}$ are denoted by $E^\infty_{p,q}$. In this case, the bigraded module with components $E^\infty_{p,q}$ is called the $\infty$-page of the spectral sequence. If $E^N=E^\infty$ for some finite $N$, the spectral sequence is said to \emph{collapse} on the $N$-th page.

Each successive page of the spectral sequence provides a successively better approximation of the homology of the total complex, so if the spectral sequence of a double complex $M$ converges, it is said to converge to $\Hg_*(\Tot(M))$. In practice, this means that $\Hg_*(\Tot(M))$ may be reconstructed from the $E^\infty$ page. In particular, if $E^\infty$ consists of free modules, $\Hg_n(\Tot(M))$ is isomorphic to $\bigoplus_{p+q=n}E^\infty_{p,q}$. Generally, however, the relation between $\Hg_*(\Tot(M))$ and $E^\infty$ is  slightly more complicated. If a spectral sequence converges to $\Hg_*(\Tot(M))$, then there exists a filtration
\[
\Hg_{p+q}(\Tot(M))^0 \subseteq \Hg_{p+q}(\Tot(M))^1 \subseteq \ldots \subseteq\Hg_{p+q}(\Tot(M))^p \subseteq \ldots\subseteq \Hg_*(\Tot(M))
\]
and the $E^\infty$ consists of successive quotients of various steps of the filtration of $\Hg_*(\Tot(M))$ arising from the structure of the double complex:
\[
E^\infty_{p,q}\cong\frac{\Hg_{p+q}(\Tot(M))^p}{\Hg_{p+q}(\Tot(M))^{p-1}}.
\]
Note that the $p$ here denotes the position in the filtration which coincides with the column of the double complex. It is straightforward to check that in our case, the spectral sequences are convergent.

Hence, reconstructing $\Hg_*(\Tot(M))$ up to isomorphism from $E^\infty$ in general requires us to solve a series of extension problems over each antidiagonal $p+q=n$. In the case we're interested in, the filtration has two additional properties which  follow from the explicit description in the Appendix, namely
\[
\PH_n(\Tot(M))^{-1}=0\qquad\text{and}\qquad\PH_n(\Tot(M))^n=\PH_n(X).
\]

The first three steps in a spectral sequence are shown in Figure \ref{fig:threesteps}. The spectral sequence relevant to our needs is called the Mayer-Vietoris spectral sequence. It is a first quadrant spectral sequence, meaning that $E^r_{p,q}=0$ if either $p<0$ or $q<0$. Note that first quadrant spectral sequences always converge, since eventually all differentials beginning or ending at a particular $(p,q)$ in the first quadrant will point outside this quadrant. The Mayer-Vietoris spectral sequence is defined as the spectral sequence of a particular double complex arising from a cover of the space whose homology we are interested in.


\begin{center}
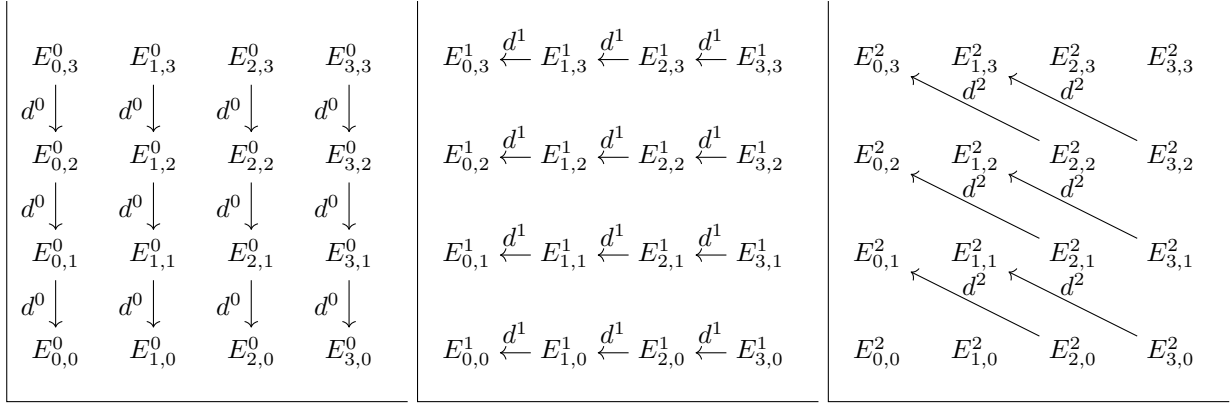
\begin{figure}[h]
\begin{tikzpicture}[scale=1.3]
  \foreach \x in {0,...,3}
    \foreach \y in {0,...,3} 
	{
		\node (E\x\y) at (\x,\y) {$E^0_{\x,\y}$};
	}
   \foreach \x/\xx in {0/1,1/2,2/3}
    \foreach \y in {0,...,3} 
	{
		\draw[->] (E\y\xx) --node[left]{$d^0$} (E\y\x);
	}

	\draw (-0.5,3.6) -- (-0.5,-0.5) -- (3.6,-0.5);
\end{tikzpicture}
\begin{tikzpicture}[scale=1.3]
  \foreach \x in {0,...,3}
    \foreach \y in {0,...,3} 
	{
		\node (E\x\y) at (\x,\y) {$E^1_{\x,\y}$};
	}
   \foreach \x/\xx in {0/1,1/2,2/3}
    \foreach \y in {0,...,3} 
	{
		\draw[->] (E\xx\y) --node[above]{$d^1$} (E\x\y);
	}

	\draw (-0.5,3.6) -- (-0.5,-0.5) -- (3.6,-0.5);
\end{tikzpicture}
\begin{tikzpicture}[scale=1.3]
  \foreach \x in {0,...,3}
    \foreach \y in {0,...,3} 
	{
		\node (E\x\y) at (\x,\y) {$E^2_{\x,\y}$};
	}
   \foreach \x/\xx in {0/2,1/3}
    \foreach \y/\yy in {0/1,1/2,2/3} 
	{
		\draw[->] (E\xx\y) --node[above]{$d^2$} (E\x\yy);
	}

	\draw (-0.5,3.6) -- (-0.5,-0.5) -- (3.6,-0.5);
\end{tikzpicture}
\caption{\label{fig:threesteps}The differentials for the first three pages of the spectral sequence. In each case, to compute the next page we take homology with respect to the differential on the current page. We set $d^0 = \partial^0$ and $d^1$ is the homomorphism induced by $\partial^1$.}
\end{figure}
\end{center}
Now, suppose we are given a pair $(\Xs,\XCov)$, where $\Xs$ is a filtered simplicial complex and $\XCov=(\XCv_i,\Filt_i)_{i\in\Lambda}$ is a filtered cover of $\Xs$ by subcomplexes. To any such pair, we may associate a commutative double complex $(E^0,\partial^0,\partial^1)$ where the underlying bigraded module is given by (recall that $\XCv_I$ is given the filtration $\Filt_I$)
\[
E^0_{p,q}=\bigoplus_{|I|=p+1}\PC_q(\XCv_I)
\]
and the boundary maps $\partial^0_{p,q}:E^0_{p,q}\to E^0_{p,q-1}$ and $\partial^1_{p,q}:E^0_{p,q}\to E^0_{p-1,q}$ are defined on the simplices by
\[
\partial^0_{p,q}(\sigma,I)=\sum_{k=0}^q(-1)^k t^{\deg(\sigma,I)-\deg(\sigma_k,I)}(\sigma_k,I)\qquad\text{and}\qquad\partial^1_{p,q}(\sigma,I)=\sum_{l=0}^p(-1)^l t^{\deg(\sigma,I)-\deg(\sigma,I_l)} (\sigma,I_l).
\]
These formulae require some explanation. To simplify things, we choose total orderings on the set $V$ of vertices of $\Xs$ and the index set $\Lambda$ of the cover $\XCov$. Note that $\Lambda$ is the set of vertices of the nerve $\Nrv$ of $\XCov$. These total orders of $V$ and $\Lambda$ allow us to speak unambiguously of ``the $k$-th vertex of $\sigma$'' and ``$l$-th vertex of $I$''. The simplices are denoted as pairs $(\sigma,I)$ to distinguish between two copies of $\sigma$ corresponding to different summands in $E^0_{p,q}$. Each simplex $\sigma=\{v_0,\ldots,v_q\}\in\XCv_I$ has a birth time $\deg(\sigma,I)$ in the filtration of $\XCv_I$. As usual, $\sigma_k:=\sigma-\{v_k\}$ are the faces of codimension $1$ in $\sigma$. Since $I=\{i_0,\ldots,i_p\}$ is a $p$-simplex in the nerve, it also makes sense to think of $I_l:=I-\{i_l\}$ as the faces of codimension $1$ in $I$.

It is a standard fact that $E^0$ is indeed a chain complex with respect to $\partial^0$ and $\partial^1$. Furthermore, $\partial^0$ and $\partial^1$ commute, since the first only operates on the chains of $\Xs$, whereas the second operates on the chains of the nerve $\Nrv$. Hence, replacing $\partial^0$ by $(-1)^p\partial^0$ yields a double complex. The spectral sequence $(E^r,d^r)$ associated to this double complex is called {\em the Mayer-Vietoris spectral sequence of $(\Xs,\XCov)$}.

This double complex is designed so that its homology is precisely the homology of $(\Xs,\Filt)$, implying the following fact, which is the main reason for the importance of the Mayer-Vietoris spectral sequence.

\begin{theorem}\label{convergence}
The Mayer-Vietoris spectral sequence of $(\Xs,\XCov)$ converges to $\PH_*(\Xs)$.
\end{theorem}
This result can be found in~\cite{lipsky2011spectral} and variations can be found in~\cite{bott2013differential,brown2012cohomology,godement1958topologie}. For completeness, we include the idea of proof in the Appendix.
%
The first page of the Mayer-Vietoris spectral sequence can be expressed as follows. Note that $d^0$ is simply $\partial^0$, which acts on each summand as the simplicial boundary operator, therefore
\begin{equation}
\label{E1}
E^1_{p,q}=\bigoplus_{|I|=p+1}\PH_q(\XCv_I).
\end{equation}
The boundary map $d^1$ is induced by $\partial^1$. Explicitly, representing homology classes by cycles, we have:
\[
d^1_{p,q}\left(\left[\sum_{n=0}^N\lambda t^{\mu_n}\sigma_n\right],I\right)=\sum_{l=0}^p(-1)^l \left(\left[\sum_{n=0}^N\lambda_n t^{\mu_n+\deg(\sigma_n,I)-\deg(\sigma_n,I_l)}\sigma_n\right],I_l\right).
\]
The only case we really need is $q=0$. In this case, the explicit formula can be simplified, and has the same form as that of $\partial^1_{p,0}$, the only difference being that it is defined on homology classes instead of simplices:
\[
d^1_{p,0}([v],I)=\sum_{l=0}^p(-1)^l t^{\deg(v,I)-\deg(v,I_l)} ([v],I_l).
\]
In the case of induced covers (see Remark \ref{inducedcov}), the explicit formula of $d^1_{p,q}$ has the same form as that of $\partial^1_{p,q}$ for all $q$. Computing $E^2$ is also straightforward, simply take the homology with respect to $d^1$. The higher pages require us to compute the higher differentials, which usually requires more work.

For illustrative purposes we now prove the Persistent Nerves Theorem of Sheehy \cite[Theorem 6]{sheehy2012multicover} using spectral sequences. In \cite{sheehy2012multicover} this is proved by using the Persistent Nerve Lemma of Chazal and Oudot \cite[Lemma 3.4.]{chazal2008towards}. Our proof also uses the idea of Chazal and Oudot, namely the fact that the Mayer-Vietoris blowup complex associated to $(\Xs,\XCov)$ is homotopy equivalent to $\Xs$ is used to establish Theorem \ref{convergence} (see Appendix). 
Our theorems are motivated by and can be thought of as a generalization of this proof (recall that $\PH_*(\cdot)$ denotes persistent homology). We begin with a preliminary Lemma.
\begin{lemma}\label{induced}
Suppose the chain complexes $(\Cg',\partial')$ and $(\Cg'',\partial'')$ are $\eps$-interleaved as chain complexes. Then their homologies $\Hg'_*=\PH_*(\Cg')$ and $\Hg''_*=\PH_*(\Cg'')$ are $\eps$-interleaved as graded modules.
\end{lemma}

\begin{proof}
Let $\phi:\Cg'\to\Cg''$ and $\psi:\Cg''\to\Cg'$ be the interleaving maps. Since $(\phi,\psi)$ is an interleaving of chain complexes, $\phi$ and $\psi$ preserve cycles and boundaries. Therefore the restrictions $\phi_\Zg:\Zg'\to\Zg''$ and $\psi_\Zg:\Zg''\to\Zg'$ of the interleaving maps define an $\eps$-interleaving $(\phi_\Zg,\psi_\Zg)$ of the cycle modules, and the restrictions $\phi_\Bg:\Bg'\to\Bg''$ and $\psi_\Bg:\Bg''\to\Bg'$ provide an $\eps$-interleaving $(\phi_\Bg,\psi_\Bg)$ of the boundary modules. These also descend to the level of quotients, i.e. we may define an $\eps$-interleaving $(\phi_\Hg,\psi_\Hg)$ of $\Hg'_*$ and $\Hg''_*$ by the formulae
\[
\phi_\Hg([x])=[\phi_\Zg(x)]\qquad\text{and}\qquad\psi_\Hg([x])=[\psi_\Zg(x)].
\]
It is readily verified that these maps are well-defined and provide the appropriate interleaving.
\end{proof}
%
%
%
%
This leads us immediately to the Persistent Nerves Theorem.
\begin{theorem}
Suppose $\Xs$ is a filtered simplicial complex and $\XCov$ a persistently acyclic filtered cover of $\Xs$. Then, $\PH_*(\Xs)\cong\PH_*(\Nrv(\XCov))$.
\end{theorem}

\begin{proof}
We use the Mayer-Vietoris spectral sequence $E$ associated to $(\Xs,\XCov)$. Let $(\PC,\partial)$ be the simplical chain complex of the nerve $\Nrv$ of $\XCov$. The boundary operator is given by the explicit formula
\[
\partial_p(I)=\sum_{l=0}^p(-1)^l t^{\deg I-\deg I_l} I_l.
\]
This has the same form as the boundary operators $d^1$ in the bottom row of $E^1$ and $\partial^1$ in the bottom row of the double complex. In particular,
\[
d^1_{p,0}([v],I)=\sum_{l=0}^p(-1)^l t^{\deg(v,I)-\deg(v,I_l)} ([v],I_l).
\]
In fact, $(\PC,\partial)$ and $(E^1_{*,0},d^1_{*,0})$ are isomorphic as chain complexes, the inverse isomorphisms $\phi_p:E^1_{p,0}\to\PC_p$ and $\psi_p:\PC_p\to E^1_{p,0}$ being given by
\[
\phi_p([v],I)=t^{\deg v-\deg I}I\qquad\text{and}\qquad\psi_p(I)=([v_I],I),
\]
where we choose a vertex $v_I\in V$ with the property $\deg v_I=\deg I$. Note that $\psi$ is well-defined, because $\XCv_I$ is acyclic: if $v\neq v_I$ is another vertex with $\deg v = \deg I$, it belongs to the same homology class as $v_I$. That $\phi$ and $\psi$ are inverse to each other follows by direct computation.

By Lemma \ref{induced}, this implies that $E^2_{*,0}\cong\PH_*(\Nrv)$. Using the fact that all $\XCv_I$ are acyclic, we have that $E^1_{p,q}=0$ for $q>0$, so the higher differentials $d^r$ for $r>1$ are all trivial and therefore $E^2\cong E^{\infty}$. As all modules above the bottom row are zero, there are also no extension problems, so the conclusion follows. In more detail, by Theorem \ref{convergence}, there is a filtration $(\PH_*(\Xs)^p)_{p\in\ZZ}$, defined on $\PH_*(\Xs)$, such that
\[
E^\infty_{p,q}\cong\frac{\PH_{p+q}(\Xs)^p}{\PH_{p+q}(\Xs)^{p-1}}=\begin{cases}0;&q\neq 0,\\
\PH_{p+q}(\Nrv);&q=0.
\end{cases}
\]
Applying the third isomorphism theorem $n$ times and recalling that $\PH_n(\Xs)^{-1}=0$ and $\PH_n(\Xs)^n=\PH_n(\Xs)$, this implies
\[
\PH_n(\Xs)\cong\frac{\PH_n(\Xs)^n}{\PH_n(\Xs)^{-1}}\cong\ldots\cong\frac{\PH_n(\Xs)^n}{\PH_n(\Xs)^{n-2}}\cong\frac{\PH_n(\Xs)^n}{\PH_n(\Xs)^{n-1}}\cong E^\infty_{n,0}\cong\PH_n(\Nrv),
\]
as desired.
\end{proof}

\begin{center}
\begin{figure}
\includegraphics[width=0.7\textwidth]{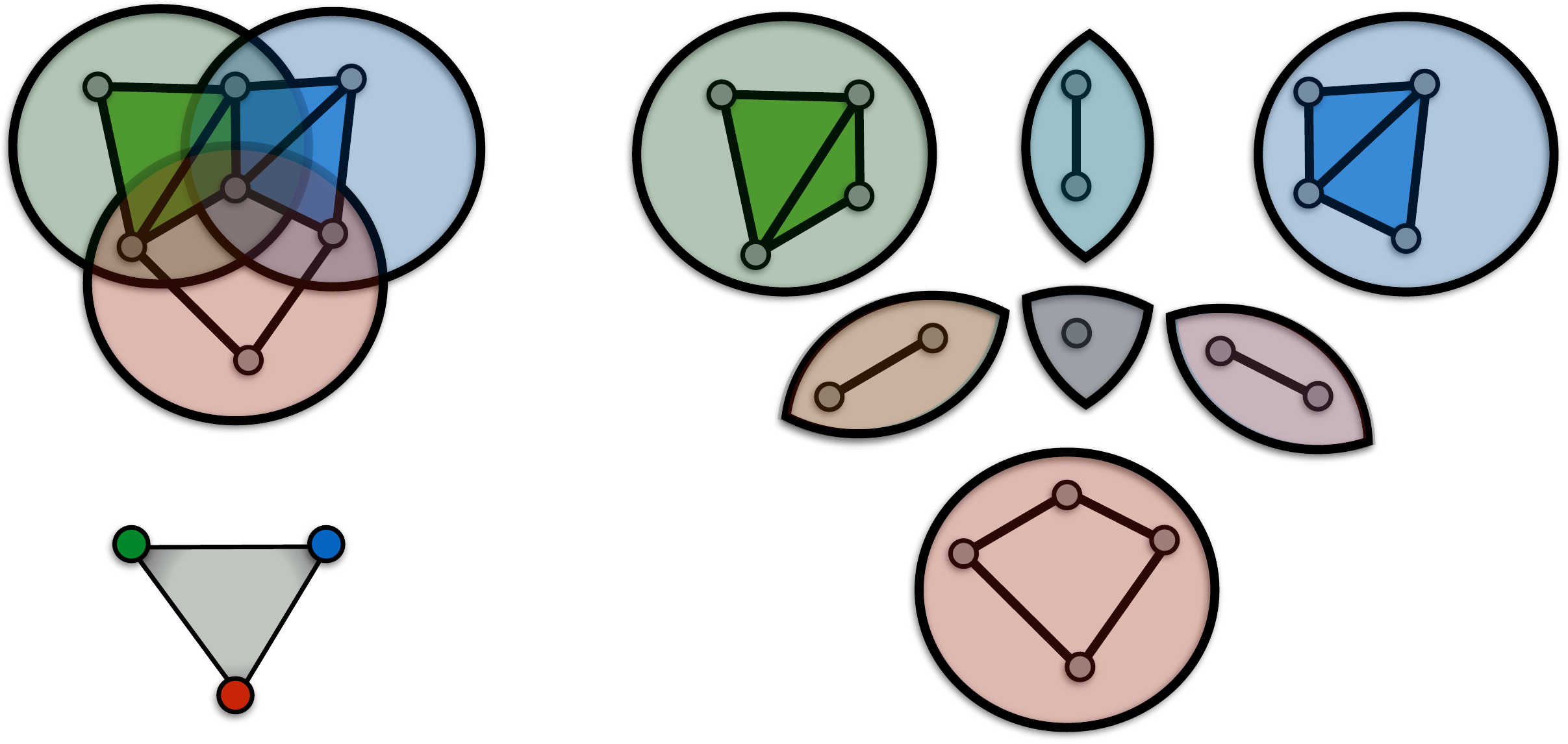}
\caption{An example construction. On the left, we have a simplicial complex which is covered by a cover with three elements. The corresponding nerve is shown below it as it is a triangle. This is {\bf not} an example of a good cover. The 0-th column of the double complex consists of a direct sum of the subcomplexes which lie in each individual element, the 1st column, the subcomplexes which lie in the pairwise intersections and the finally the 2nd column contains the triple intersection. The row index represents the dimension grading from the underlying complex, i.e. vertices in the 0-th row, edges in the 1st  row and triangles in the 2nd. Note that the total complex has potentially multiple copies of a simplex and is much larger than the original complex.}
\end{figure}
\end{center}


\section{$\eps$-Acyclic Covers}\label{sec:epscover}

Here we introduce the notion of an $\eps$-acyclic cover. For convenience, we find it easier to work with the notion of interleaving and modules rather than persistence diagrams. However, we also include the diagrams for the definitions, for completeness and to help with intuition.

In the classical setting of Theorem \ref{thm:acyclic}, we assume that each non-empty finite intersection $\XCv_I$ has the homology of a point. In our case, we wish to assume that the homology of each non-empty intersection $\XCv_I$ is $\eps$-close to the homology of a point, specifically, we require that the two homologies are $\eps$-interleaved.

To be more precise, for each $a\in\ZZ$, we define $\pt_a$ to be the $\ZZ$-filtered simplicial complex consisting of a single point, with the filtration defined by the requirement that $\pt_a^j=\emptyset$ for $j<a$ and $\pt_a^j=\{*\}$ for $j\geq a$.

\begin{definition}
A non-empty $\ZZ$-filtered simplicial complex $\Xs$ is {\em (persistently) acyclic} if it has the persistent homology of a point, i.e. $\PH_*(\Xs)\cong\PH_*(\pt_a)$ for some $a\in\ZZ$. It is {\em $\eps$-acyclic} if its persistent homology is $\eps$-interleaved with the persistent homology of a point, i.e. $\PH_*(\Xs)\itl{\eps}\PH_*(\pt_a)$ for some $a\in\ZZ$.
\end{definition}

In other words, $\eps$-acyclicity means that $\PH_q(\Xs)\itl{\eps}0$ for $q\neq 0$ and $\PH_0(\Xs)\itl{\eps}t^a\kk[t]$ for some $a\in\ZZ$. A persistence module $M$ that is $\eps$-close to the trivial module $0$, i.e. $M\itl{\eps}0$ is said to be {\em $\eps$-trivial}. The same understanding applies to persistence diagrams.

The persistence diagram of an acyclic complex consists of only the diagonal in degrees other than $0$, and a single point of the form $(a,\infty)$ in degree $0$ representing the essential class (corresponding to the first component that appears), while the persistence diagram of an $\eps$-acyclic complex consists of points which are at most $\eps$-away from the diagonal (see Figure~\ref{fig:emptytrivial}) and a single point $(a,\infty)$ in degree $0$.

\begin{center}
\begin{figure}
\begin{tikzpicture}[scale=4]
\draw[<->] (1.1,0) -- (0,0) -- (0,1.1);
\draw[fill=gray] (0,0) -- (1,1) -- (1,0);																			
\end{tikzpicture}
\hspace{1cm}
\begin{tikzpicture}[scale=4]
\draw[<->] (1.1,0) -- (0,0) -- (0,1.1);
\draw[fill=gray] (0,0) -- (1,1) -- (1,0);	
\draw[blue] 	(0,0) -- (1,1);																	
\draw[fill=blue,opacity=0.3] (0,0)--(0,0.1)-- (0.9,1) -- (1,1);

\end{tikzpicture}
\caption{\label{fig:emptytrivial}On the left we have a trivial persistence diagram and on the right an $\eps$-trivial persistence diagram, where points can occur with any multiplicity within the shaded region. }
\end{figure}
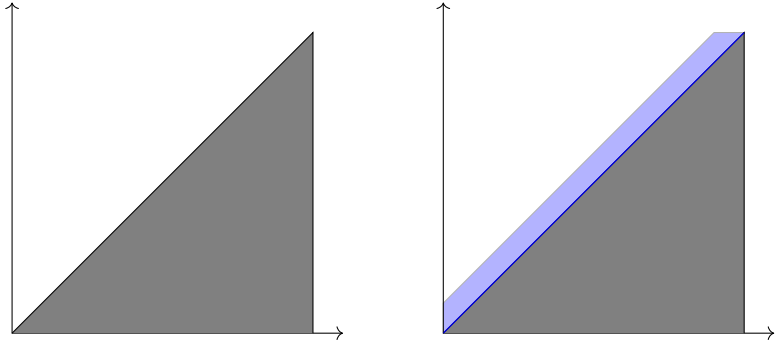
\end{center}

We can now define an $\eps$-acyclic cover. 

\begin{definition}
Let $\eps\in\NN_0$. We say that the filtered cover $\XCov$ of $\Xs$ is an {\em $\eps$-acyclic cover} if for each $I\in\Nrv(\XCov)$ there is an $a\in\ZZ$ such that $\PH_*(\XCv_I)\itl{\eps}\PH_*(\pt_a)$.
\end{definition}
Assuming $\XCov$ is an $\eps$-acyclic cover of $\Xs$, our aim is to prove that $\PH_*(\Xs)$ and $\PH_*(\Nrv)$ are $\eta$-interleaved, where $\eta$ is bounded above in terms of $\eps$ and possibly some other parameter. To help with intuition, we now relate the double complex we use in the spectral sequence with the notion of an $\eps$-acyclic cover. This is best expressed in terms of the $E^1$ pages. The $E^1$ page of an acyclic cover and an $\eps$-acyclic cover are shown in Figure ~\ref{fig:acyclicpage}. For $q>0$, the elements are $0$ or $\eps$-interleaved with $0$ respectively. For $q=0$, each element or non-empty intersection yields one essential class. Since
\[
E^1_{p,0} = \bigoplus_{|I|=p+1}\PH_0(\XCv_I),
\]
the meaning of $\eps$-acyclicity is that each element of the $E^1$-page is either an essential class corresponding to some $I\in\Nrv(\XCov)$ or $\eps$-trivial.

\begin{center}
\begin{figure}
\begin{tikzpicture}[xscale=2.3, yscale=1.5]
  \foreach \x in {0,...,2}
    \foreach \y in {1,...,2} 
	{
		\node (E\x\y) at (\x,\y) {$0$};
	}
 	\foreach \x in {0,...,2}{
 		\node (E\x0) at (\x,0) {$E^1_{\x,0}$};
 	}
	\draw (-0.5,2.6) -- (-0.5,-0.5) -- (2.6,-0.5);
	\draw[fill=blue,opacity=0.3] (-0.5,-0.5) rectangle (2.6,0.5);
	
	
\end{tikzpicture}
%
\hspace{1cm}
%
\begin{tikzpicture}[xscale=2.3, yscale=1.5]
  \foreach \x in {0,...,2}
    \foreach \y in {1,...,2} 
	{
		\node (E\x\y) at (\x,\y) {$E^1_{\x,\y} \itl{\eps} 0$};
	}
 	\foreach \x in {0,...,2}{
 		\node (E\x0) at (\x,0) {$E^1_{\x,0} $};
	
 	}
	\draw (-0.5,2.6) -- (-0.5,-0.5) -- (2.6,-0.5);
	\draw[fill=blue,opacity=0.3] (-0.5,-0.5) rectangle (2.6,0.5);
	
	
\end{tikzpicture}
\caption{\label{fig:acyclicpage}The $E^1$ page of an acyclic cover (left) and an $\eps$-acyclic cover (right). In the case of an acyclic cover, the spectral sequence degenerates on the $E^2$ page because the non-trivial terms are concentrated in the first row. For the $\eps$-acyclic cover, the terms above the first row are only required to be $\eps$-trivial.}
\end{figure}
\end{center}

The notion of $\eps$-acyclicity need only hold at the level of homology, or equivalently, the interleaving is defined on the $E^1$-page of the spectral sequence. Consider the corresponding condition at the chain level, i.e. the cover is interleaved with an acyclic cover at the chain level. This implies that the terms on the $E^0$-page are $\eps$-interleaved. It is straightforward to check that an interleaving on the $E^0$-page induces an interleaving on the total complex and hence on the persistent homology. This observation combined with the lower bounds presented in Section \ref{sec:lowerbnd} illustrates that $\eps$-acyclicity is a strictly weaker requirement than requiring chain level interleaving as well as that in certain natural cases, chain level interleavings do not exist. 

\subsection{Construction}
Here we describe an explicit construction of the filtration for the nerve. Recall the standard construction for the nerve (Definition~\ref{def:nerve}).
In our case however, the cover elements $\XCv_i$ are filtered by functions $f_i$, and the space $\Xs$ has a filtration as well, given by $f=\min_{i\in\Lambda} f_i$. Therefore, we must also describe a function $g$ on the nerve. One natural construction is the following. For $I\in\Nrv$, 
define
\begin{equation}\label{eq:definitionofg}
g(I) = \min\{j\mid\XCv_I^j\neq\emptyset\}.
\end{equation}
That is, we place a simplex in the filtration, the first time the intersection is not empty. Note there are numerous other constructions, such as taking the average or maximum value which may make more sense in certain cases. It is clear that the sublevel sets of $g$ define a filtration on the nerve.

\section{Left and Right Interleaving}\label{sec:interleaving}
 We extend the usual notion of interleaving to left and right interleaving.  This is a refinement of interleaving and certain structural properties will be useful for proving our main result. Readers may skip this section and replace the notions of left and right interleaving in Section \ref{sec:thms} simply by interleaving, since this is all that is needed for the easy result (Theorem \ref{thm:easy}). The main result in this section is Proposition ~\ref{doesnotaddup}. We note that a similar result could be obtained using the techniques in \cite{bauer2014induced} by considering matchings between barcodes. One drawback of using matchings is that it requires the persistence module to be pointwise finite dimensional~\cite{bauer2014induced} or at least have an interval decomposition. Our alternative approach has no such requirement, as it applies to modules where no such decomposition exists. 

 This represents a new viewpoint on interleavings since left and right interleavings are asymmetric leading to several different types of composition (addressed in Proposition~\ref{doesnotaddup}). Though we only use one type of composition in Section~\ref{sec:thms}, the others are included for completeness as well as to highlight an interesting phenomenon. We show that for most types of composition of right and left interleavings, the factors are not additive but rather take the maximum of the two component interleavings. Except for one specific case, this holds for 
 more general persistence theories such as persistence over $\mathbb{Z}$ \cite{patel2016semicontinuity} and with appropriate modification, to multidimensional persistence \cite{lesnick2015theory}. The exception is the fourth case in Proposition~\ref{doesnotaddup}, which has the additional requirement of having projective dimension $1$. Unfortunately, this is precisely the case used in Section~\ref{sec:thms}. We conjecture that this is not an artifact of the proof technique but that the statement does not hold for this type of composition in the case of more general persistence modules. If so, we believe this asymmetry is of independent interest. Finally, we show an equivalence between a general interleaving and a sequence of right and left interleavings. This decomposition can be interpreted as ``shortening'' and ``lengthening'' bars, but holds even when a barcode does not exist.  

 In order to prove our result, we must work with approximations of persistence modules efficiently. In particular, we must be able to estimate kernels and cokernels of maps. Intuitively, given a map whose codomain is approximately zero, the kernel should be approximately equal to the domain. The following proposition justifies this intuition. We remind the reader that ``morphism'' always means $0$-morphism, i.e. it is assumed that degrees are preserved.

\begin{proposition}\label{kernel}
Let $g:N\to P$ be a morphism of $\kk[t]$-modules and $P\itl{\eps}0$. Then, $N\itl{2\eps}\ker g$. In fact, $\phi:N\to\ker g$ and $\psi:\ker g\to N$ defined by $\phi(n)=t^{2\eps}n$ and $\psi(m)=m$ satisfy $\phi\psi=\id_{2\eps}$ and $\psi\phi=\id_{2\eps}$.
\end{proposition}

\begin{proof}
The equalities follow directly from the definitions of $\phi$ and $\psi$. Therefore, $(\phi,\id_{2\eps}\psi)$ is a $2\eps$-interleaving. We only have to verify that $\phi$ is well-defined. To see this, note that $P\itl{\eps}0$, so multiplication by $t^{2\eps}$ is the zero map on $P$. This means that for any $n\in N$, we have $t^{2\eps}n\in\ker g$, because $g(t^{2\eps}n)=t^{2\eps} g(n)=0$.
\end{proof}

The analogous statement for cokernels is also true by the dual argument.

\begin{proposition}\label{cokernel}
Let $f:M\to N$ be a morphism of $\kk[t]$-modules and $M\itl{\eps}0$. Then, $N\itl{2\eps}\coker f$. In fact, $\eta:N\to\coker f$ and $\theta:\coker f\to N$ defined by $\eta(n)=[n]$ and $\theta([n])=t^{2\eps}n$ satisfy $\eta\theta=\id_{2\eps}$ and $\theta\eta=\id_{2\eps}$.
\end{proposition}

\begin{proof}
Again, the two equalities follow directly from the definitions and $(\id_{2\eps}\eta,\theta)$ is a $2\eps$-interleaving. We only have to verify that $\theta$ is well-defined. To see this, note that $M\itl{\eps}0$, so $t^{2\eps}M=0$ and thus $t^{2\eps}\im f=0$. Now suppose $[n_1]=[n_2]$. It follows that $n_1-n_2\in\im f$, so $t^{2\eps}(n_1-n_2)=0$, concluding the proof.
\end{proof}

As described in Section \ref{sec:modint}, interleavings define a metric between modules. It turns out, however, that the interleavings arising in these two situations have somewhat special properties, so they deserve separate definitions to distinguish them from ordinary interleavings. We will exploit the properties of such interleavings to obtain tight bounds in the Approximate Nerve Theorem.

\begin{definition}\label{def:linterleave}
Suppose $M$ and $N$ are $\kk[t]$-modules. We say that $M$ and $N$ are {\em $2\eps$-left interleaved} and write $M\Litl{2\eps}N$ if there is a $\kk[t]$-module $P\itl{\eps}0$ and a short exact sequence of the form $0\to M\to N\to P\to 0$.
\end{definition}

\begin{definition}\label{def:rinterleave}
Suppose $N$ and $P$ are $\kk[t]$-modules. We say that $N$ and $P$ are {\em $2\eps$-right interleaved} and write $N\Ritl{2\eps} P$ if there is a $\kk[t]$-module $M\itl{\eps}0$ and a short exact sequence of the form $0\to M\to N\to P\to 0$.
\end{definition}

\begin{remark}\label{nonsymmetric}
Note that these definitions are not symmetric, i.e. $M\Litl{2\eps}N$ does not imply $N\Litl{2\eps}M$ and $N\Ritl{2\eps}P$ does not imply $P\Ritl{2\eps}N$. To see the asymmetry, let $M$ consist of one generator born at $j=0$ with one relation at $j=a+2\eps$, and $N$ have one generator born at $j=0$ with one relation at $j=a$. The kernel of the obvious map is $\eps$-interleaved with 0, hence $M$ and $N$ are $\eps$-left interleaved. In fact there exists no $0$-morphism from $N\rightarrow M$, hence their right or left interleaving distance is infinite.
\end{remark}

We now prove some properties of left and right interleavings. As mentioned above, left and right interleavings are not symmetric, but they do still satisfy the triangle inequality. Positive definiteness also holds, but this is easy to see by definition -- simply take the $0$ morphism.

Before continuing, we establish a basic  proposition, similar in spirit to Proposition \ref{kernel} and Proposition \ref{cokernel}. 

\begin{proposition}\label{extension}
Suppose we are given an exact sequence
\[
0\rightarrow M\xrightarrow{i} N\xrightarrow{f} P\rightarrow 0
\]
where $M\itl{\eps_1}0$ and $P\itl{\eps_2}0$. Then $N\itl{\eps_1+\eps_2}0$.
\end{proposition}

\begin{proof}
We need to show that $t^{2(\eps_1+\eps_2)}N=0$. Let $n\in N$. Note that $f(t^{2\eps_2}n)=t^{2\eps_2}f(n)=0$, since $t^{2\eps_2}P=0$, so $m=t^{2\eps_2}n\in\ker f=M$. Therefore $t^{2(\eps_1+\eps_2)}n=t^{2\eps_1}m=0$, since $t^{2\eps_1}M=0$.
\end{proof}

Our main motivation for introducing left and right interleavings is to study how the metrics between modules act with respect to composition. We show that
\begin{itemize}
\item the approximation factors are additive under composition of the same types of interleaving (i.e. left with left or right with right),
\item only the maximum of the approximation factors is relevant when composing different types of interleaving (i.e. left with right or right with left).
\end{itemize}
We first require some basic structural propositions. 

\begin{proposition}\label{prop:proof}
Suppose $f:M\to N$ and $g:N\to P$ are morphisms of modules. Then there is exact sequence of the form
\[
0\to\ker f\to\ker gf\to\ker g\to\coker f\to\coker gf\to\coker g\to 0
\]
\end{proposition}

\begin{proof}
Note that the diagrams
\begin{center}
\begin{tikzpicture}[scale=1.5]
\node (01) at (-1,1) {$0$};
\node (M1) at (0,1) {$M$};
\node (M2) at (1,1) {$M$};
\node (02) at (2,1) {$0$};
\node (0B) at (-2,0) {$0$};
\node (K) at (-1,0) {$\ker g$};
\node (N) at (0,0) {$N$};
\node (P) at (1,0) {$P$};
\path[->]
(01) edge (M1)
(M1) edge node[above]{$\id$} (M2)
(M2) edge (02)
(0B) edge (K)
(K) edge (N)
(N) edge node[below]{$g$} (P)
(01) edge (K)
(M1) edge node[left]{$f$} (N)
(M2) edge node[right]{$gf$} (P);
\end{tikzpicture}
\begin{tikzpicture}[scale=1.5]
\node (M) at (-1,1) {$M$};
\node (N) at (0,1) {$N$};
\node (C) at (1,1) {$\coker f$};
\node (0) at (2,1) {$0$};
\node (01) at (-2,0) {$0$};
\node (P1) at (-1,0) {$P$};
\node (P2) at (0,0) {$P$};
\node (02) at (1,0) {$0$};
\path[->]
(M) edge node[above]{$f$} (N)
(N) edge (C)
(C) edge (0)
(01) edge (P1)
(P1) edge node[below]{$\id$} (P2)
(P2) edge (02)
(M) edge node[left]{$gf$} (P1)
(N) edge node[right]{$g$} (P2)
(C) edge (02);
\end{tikzpicture}
\end{center}
have exact rows. Applying the Snake Lemma to each of these diagrams, we obtain exact sequences
\[
0\to\ker f\to\ker gf\to\ker g\to\coker f\to\coker gf
\]
and
\[
\ker gf\to\ker g\to\coker f\to\coker gf\to\coker g\to 0
\]
By construction, the two maps $\ker g\to\coker f$ are actually the same, so splicing the two sequences yields
\[
0\to\ker f\to\ker gf\to\ker g\to\coker f\to\coker gf\to\coker g\to 0
\]
as desired.
\end{proof}

This immediately yields two useful corollaries, dual to each other.

\begin{corollary}\label{cokerexact}
Suppose $f:M\to N$ and $g:N\to P$ are morphisms of modules with $g$ injective. Then the sequence
\[
0\to\coker f\to\coker gf\to\coker g\to 0
\]
is exact.
\end{corollary}

\begin{corollary}\label{kerexact}
Suppose $f:M\to N$ and $g:N\to P$ are morphisms of modules with $f$ surjective. Then the sequence
\[
0\to\ker f\to\ker gf\to\ker g\to 0
\]
is exact.
\end{corollary}

Using these, we may now prove the triangle inequality for left-interleavings.

\begin{proposition}\label{prop:leftright}
Suppose $M\Litl{2\eps_1} N$ and $N\Litl{2\eps_2} P$. Then $M\Litl{2(\eps_1+\eps_2)} P$.
\end{proposition}

\begin{proof}
The assumptions mean that we have exact sequences
\[
0\rightarrow M\xrightarrow{i} N\xrightarrow{f}\coker i\rightarrow 0\qquad\text{and}\qquad 0\rightarrow N\xrightarrow{j} P\xrightarrow{g} \coker j\rightarrow 0
\]
with $\coker i\itl{\eps_1} 0$ and $\coker j\itl{\eps_2} 0$. Since $j$ is injective, we have
\[
0\to\coker i\to\coker ji\to\coker j\to 0
\]
by Corollary \ref{cokerexact}, so $\coker ji\itl{\eps_1+\eps_2}0$ by Proposition \ref{extension}. Observing that the sequence
\[
0\rightarrow M\xrightarrow{ji} P\rightarrow\coker ji\rightarrow 0
\]
is exact completes the proof.
\end{proof}


The same result holds for right-interleavings. 

\begin{proposition}\label{prop:rightleft}
Suppose $M\Ritl{2\eps_1} N$ and $N\Ritl{2\eps_2} P$. Then $M\Ritl{2(\eps_1+\eps_2)} P$.
\end{proposition}

\begin{proof}
By the assumptions there are exact sequences
\[
0\rightarrow \ker f\xrightarrow{i} M\xrightarrow{f} N\rightarrow 0\qquad\text{and}\qquad 0\rightarrow \ker g\xrightarrow{j} N\xrightarrow{g} P\rightarrow 0
\]
with $\ker f\itl{\eps_1}0$ and $\ker g\itl{\eps_2}0$. Since $f$ is surjective, we have
\[
0\to\ker f\to\ker gf\to\ker g\to 0
\]
by Corollary \ref{kerexact}, so $\ker gf\itl{\eps_1+\eps_2}0$ by Proposition \ref{extension}. Observing that the sequence
\[
0\rightarrow\ker gf\rightarrow M\xrightarrow{gf} P\rightarrow 0
\]
is exact completes the proof.
\end{proof}

The previous results are required to show that the interleavings are in a sense closed under composition, e.g. composing two left interleavings (with a suitable ordering of terms), yields a left interleaving (with an additive approximation factor). Now, we show the more interesting property: most combinations of the different notions of interleavings do not interact, i.e. composition yields the maximum of the two rather than an additive factor. First, we show that if composition is {\bf not} in the natural order as in Proposition~\ref{prop:rightleft} and \ref{prop:leftright}, the interleavings do not yield an additive factor.
\begin{proposition}\label{prop:nonnatural_in}
Suppose one of the following two possibilities holds,
\[
M\Litl{2\eps} N \quad \mbox{and}\quad P\Litl{2\eps} N, \qquad\mbox{or} \qquad M\Ritl{2\eps} N  \quad \mbox{and}\quad  P\Ritl{2\eps} N,
\]
then $M\itl{2\eps}P$.
\end{proposition}
\begin{proof}
In the first case, we have the following two exact sequences:
\[
0\rightarrow M\xrightarrow{i} N\xrightarrow{f}X\rightarrow 0\qquad\text{and}\qquad 0\rightarrow P\xrightarrow{g} N\xrightarrow{j} Y \rightarrow 0
\]
Similarly in the second case, we have:
\[
0\rightarrow X \xrightarrow{f}M\xrightarrow{i} N\rightarrow 0\qquad\text{and}\qquad 0\rightarrow Y  \xrightarrow{j}P\xrightarrow{g} N  \rightarrow 0
\]
In both cases, by assumption $X\itl{\eps}0$ and $Y\itl{\eps}0$. 
By Proposition \ref{kernel} and Proposition \ref{cokernel}, for both cases there are $2\eps$-interleavings $(\phi,\psi)$ of $M$ and $N$ and $(\eta,\theta)$ of $P$ and $N$. These fit into the following commutative diagram, where the horizontal arrows are ordinary morphisms and all other arrows are $2\eps$-morphisms.
\begin{center}
\begin{tikzpicture}[scale=2]
\node (MT) at (-1,1) {$M$};
\node (NT) at (0,1) {$N$};
\node (PT) at (1,1) {$P$};
\node (M) at (-1,0) {$M$};
\node (N) at (0,0) {$N$};
\node (P) at (1,0) {$P$};
\node (MB) at (-1,-1) {$M$};
\node (NB) at (0,-1) {$N$};
\node (PB) at (1,-1) {$P$};
\path[->]
(MT) edge node[above]{$i$} (NT)
(PT) edge node[above]{$g$} (NT)
(M) edge node[below]{$i$} (N)
(P) edge node[below]{$g$} (N)
(MB) edge node[below]{$i$} (NB)
(PB) edge node[below]{$g$} (NB)
(MT) edge node[left]{$t^{2\eps}$} (M)
(NT) edge node[right]{$t^{2\eps}$} (N)
(PT) edge node[right]{$t^{2\eps}$} (P)
(NT) edge node[below]{$\psi$} (M)
(NT) edge node[below]{$\theta$} (P)
(M) edge node[left]{$t^{2\eps}$} (MB)
(N) edge node[right]{$t^{2\eps}$} (NB)
(P) edge node[right]{$t^{2\eps}$} (PB)
(N) edge node[below]{$\psi$} (MB)
(N) edge node[below]{$\theta$} (PB);
\end{tikzpicture}
\end{center}
By inspection of this diagram, we see that $(\theta i,\psi g)$ is a $2\eps$-interleaving of $M$ and $P$.
\end{proof}

\begin{proposition}\label{prop:nonnatural_out}
Suppose one of the following two possibilities holds,
\[
N\Litl{2\eps} M \quad \mbox{and}\quad N\Litl{2\eps} P, \qquad\mbox{or} \qquad N\Ritl{2\eps} M  \quad \mbox{and}\quad  N\Ritl{2\eps} P,
\]
then $M\itl{2\eps}P$.
\end{proposition}
\begin{proof}
The proof is similar as above. For each case, we get two pairs of exact sequences
\[
0\rightarrow N\xrightarrow{i} M \xrightarrow{f}X\rightarrow 0\qquad\text{and}\qquad 0\rightarrow N\xrightarrow{g} P\xrightarrow{j} Y \rightarrow 0
\]
and  
\[
0\rightarrow X \xrightarrow{f}N\xrightarrow{i} M\rightarrow 0\qquad\text{and}\qquad 0\rightarrow Y  \xrightarrow{j}N\xrightarrow{g} P  \rightarrow 0
\]
with $X\itl{\eps}0$ and $Y\itl{\eps}0$. Again, by Proposition \ref{kernel} and Proposition \ref{cokernel}, we have $2\eps$-interleavings $(\phi,\psi)$ of $N$ and $M$ and $(\eta,\theta)$ of $N$ and $P$, which fit into the following commutative diagram, where the horizontal arrows are ordinary morphisms and all other arrows are $2\eps$-morphisms.
\begin{center}
\begin{tikzpicture}[scale=2]
\node (MT) at (-1,1) {$M$};
\node (NT) at (0,1) {$N$};
\node (PT) at (1,1) {$P$};
\node (M) at (-1,0) {$M$};
\node (N) at (0,0) {$N$};
\node (P) at (1,0) {$P$};
\node (MB) at (-1,-1) {$M$};
\node (NB) at (0,-1) {$N$};
\node (PB) at (1,-1) {$P$};
\path[->]
(NT) edge node[above]{$i$} (MT)
(NT) edge node[above]{$g$} (PT)
(N) edge node[below]{$i$} (M)
(N) edge node[below]{$g$} (P)
(NB) edge node[below]{$i$} (MB)
(NB) edge node[below]{$g$} (PB)
(MT) edge node[left]{$t^{2\eps}$} (M)
(NT) edge node[right]{$t^{2\eps}$} (N)
(PT) edge node[right]{$t^{2\eps}$} (P)
(MT) edge node[below]{$\psi$} (N)
(PT) edge node[below]{$\theta$} (N)
(M) edge node[left]{$t^{2\eps}$} (MB)
(N) edge node[right]{$t^{2\eps}$} (NB)
(P) edge node[right]{$t^{2\eps}$} (PB)
(M) edge node[below]{$\psi$} (NB)
(P) edge node[below]{$\theta$} (NB);
\end{tikzpicture}
\end{center}
By inspection of this diagram, we see that $(g\psi,i\theta)$ is a $2\eps$-interleaving of $M$ and $P$.
\end{proof}

We conclude with showing that all other combinations of left and right-interleavings do not interact, i.e. composing a $2\eps$-left interleaving followed by a $2\eps$-right interleaving still yields a $2\eps$-interleaving. As the two notions are not symmetric, there are four such possible cases to treat. It turns out that three of the four cases can be handled directly, while the fourth is more involved. 

\begin{proposition}\label{doesnotaddup}
Suppose one of the following four possibilities holds:
\begin{itemize}
\item $M\Litl{2\eps} N$ and $N\Ritl{2\eps} P$,
\item $N\Litl{2\eps} M$ and $N\Ritl{2\eps} P$,
\item $M\Litl{2\eps} N$ and $P\Ritl{2\eps} N$ or
\item $N\Litl{2\eps} M$ and $P\Ritl{2\eps} N$.
\end{itemize}
Then  $M\itl{2\eps}P$.
\end{proposition}
\begin{proof}
We treat each possibility separately.

{\bf First case.} We give a direct argument. There are exact sequences
\[
0\rightarrow M\xrightarrow{i} N\xrightarrow{f}X\rightarrow 0\qquad\text{and}\qquad 0\rightarrow Y\xrightarrow{j}N\xrightarrow{g} P\rightarrow 0
\]
that is $M = \ker f$ and $P = \coker j$ with $X\itl{\eps}0$ and $Y\itl{\eps}0$. The interleaving maps $\phi:M\to P$ and $\psi:P\to M$ may be defined explicitly by the formulae  $\phi(m)=t^{2\eps}g(i(m))$ and $\psi([n])=t^{2\eps}n$. Here, $[n]=g(n)$ is the class in $\coker j$ represented by $n\in N$. Note that $t^{2\eps}n\in M$, since $X\itl{\eps}0$. It is clear that $\phi$ is well-defined. To show that $\psi$ is well-defined, observe that $Y\itl{\eps}0$ implies $t^{2\eps}Y=0$ and therefore, $t^{2\eps}\im j=0$, so if $[n_1]=[n_2]$, we have $t^{2\eps}n_1=t^{2\eps}n_2$.

We remark that shifting by $2\eps$ is not necessary for the first map to be well-defined and is only done to adhere to the definition of interleaving. In fact, without this shifting we already have that $(gi)\circ\psi=\id_{2\eps}$ and $\psi\circ(gi)=\id_{2\eps}$, which is important, as it is used in the proof of Proposition \ref{prop:equivalence}.


{\bf Second case.} There are exact sequences
\[
0\rightarrow N\xrightarrow{i} M\xrightarrow{f}X\rightarrow 0\qquad\text{and}\qquad 0\rightarrow Y\xrightarrow{j}N\xrightarrow{g} P\rightarrow 0
\]
with $X\itl{\eps}0$ and $Y\itl{\eps}0$. There are  $2\eps$-interleavings $(\phi,\psi)$ of $N$ and $M$ and $(\eta,\theta)$ of $N$ and $P$, which fit into the same commutative diagram as in the proof of Proposition~\ref{prop:nonnatural_out}. Similarly, we conclude that $(g\psi,i\theta)$ is a $2\eps$-interleaving of $M$ and $P$.

{\bf Third case.} There are exact sequences
\[
0\rightarrow M\xrightarrow{i} N\xrightarrow{f}X\rightarrow 0\qquad\text{and}\qquad 0\rightarrow Y\xrightarrow{j}P\xrightarrow{g} N\rightarrow 0
\]
with $X\itl{\eps}0$ and $Y\itl{\eps}0$. There are  $2\eps$-interleavings $(\phi,\psi)$ of $N$ and $M$ and $(\eta,\theta)$ of $N$ and $P$, which fit into the same commutative diagram as in the proof of Proposition~\ref{prop:nonnatural_in}. Similarly, we conclude that $(\theta i,\psi g)$ is a $2\eps$-interleaving of $M$ and $P$.

{\bf Fourth case.} There are exact sequences
\[
0\rightarrow N\xrightarrow{i} M\xrightarrow{f}X\rightarrow 0\qquad\text{and}\qquad 0\rightarrow Y\xrightarrow{j}P\xrightarrow{g} N\rightarrow 0
\]
with $X\itl{\eps}0$ and $Y\itl{\eps}0$. To the latter, we associate the following long exact sequence of $\Ext$-modules\footnote{Note that $\Hom$-modules consist of morphisms of $\kk[t]$-modules. These are degree-preserving. The appropriate notion of $\Ext$-module needs to reflect this. In particular, the maps used in the relevant projective resolutions must also be degree-preserving.}:
\[
0\rightarrow \Hom(X,Y)\rightarrow \Hom(X,P)\rightarrow \Hom(X,N)\rightarrow \Ext(X,Y)\rightarrow \Ext(X,P)\rightarrow \Ext(X,N)\rightarrow 0
\]
Note that all higher $\Ext$-modules are $0$. To see this, recall that the projective dimension $\projdim(X)$ of a $\kk[t]$-module $X$ is the smallest $n\in\NN_0$ such that $\Ext^{n+1}(X,M)$ vanishes for all $\kk[t]$-modules $M$ (see \cite[Proposition 8.6]{rotman2008introduction}). It is known that any module over a principal ideal domain has projective dimension at most $1$, so in particular $\Ext^2(X,Y)=0$, as desired. (There is a slight subtlety here that the number $\projdim(X)$ could in principle depend on whether $X$ is regarded as a $\kk[t]$-module or a $\kk[t]_\ngr$-module. That this is not the case follows from \cite[Corollary 3.3.7]{nuastuasescu1979graded}.)

In particular, $\Ext(X,P)\rightarrow\Ext(X,N)$ is an epimorphism. Using the classical interpretation of elements of $\Ext$-modules as (equivalence classes of) extensions of modules and maps between them as morphisms of such extensions implies that there is a map of extensions
\begin{center}
\begin{tikzpicture}[scale=1.5]
\node (0TL) at (-2,1) {$0$};
\node (P) at (-1,1) {$P$};
\node (Q) at (0,1) {$Q$};
\node (XT) at (1,1) {$X$};
\node (0TR) at (2,1) {$0$};
\node (0BL) at (-2,0) {$0$};
\node (N) at (-1,0) {$N$};
\node (M) at (0,0) {$M$};
\node (XB) at (1,0) {$X$};
\node (0BR) at (2,0) {$0$};
\path[->]
(0TL) edge (P)
(P) edge (Q)
(Q) edge (XT)
(XT) edge (0TR)
(0BL) edge (N)
(N) edge (M)
(M) edge (XB)
(XB) edge (0BR)
(P) edge node[left]{$g$} (N)
(Q) edge node[left]{$h$} (M)
(XT) edge node[left]{$\id$} (XB);
\end{tikzpicture}
\end{center}
Now, using the Snake Lemma on this diagram, we see that the sequence
\[
0\to\ker g\to\ker h\to\ker\id\to\coker g\to\coker h\to\coker\id\to 0
\]
is exact. Since $\ker g=Y$ and $\coker g=\ker\id=\coker\id=0$, the sequence
\[
0\to Y\to Q\to M\to 0
\]
is exact. Therefore $Q\Ritl{2\eps} M$ and $P\Litl{2\eps} Q$, so the fourth case reduces to the first case.
\end{proof}




Intuitively, the notion of left and right interleavings corresponds to the notion of shortening (respectively lengthening) bars by changing birth and death times. This was described in \cite{bauer2014induced} using matchings. The main advantage of using short exact sequences is that the independence between modifying birth and death times can be captured without a decomposition existing. It also gives an alternative algebraic characterization of when this holds, namely that the projective dimension is one. To make this connection concrete, we prove that every interleaving admits a decomposition into left and right interleavings. We also show the converse, giving a characterization of an interleaving given a decomposition. We first require one additional definition.
\begin{definition}
If  $S$ is a persistence module, there is an $\eps$-\emph{shifted} module $S(\eps)$ which is a reparameterization of $S$ by
\[
S^\alpha(\eps) = S^{\alpha+\eps}.
\]
\end{definition}

\begin{proposition}\label{prop:equivalence}
There exists an interleaving $M\itl{2\eps}S$ if and only if $\exists N,P,Q$ such that
\begin{equation*}
\begin{aligned}[c]
M\Ritl{2\eps} N, \\
Q\Litl{2\eps} P,
\end{aligned}
\qquad 
\begin{aligned}[c]
N\Litl{2\eps} P,\\
S\Ritl{2\eps} Q.
\end{aligned}
\end{equation*}
\end{proposition}
\begin{proof}
We first show if $M$ is $2\eps$-interleaved with $S$ then $N,P,$ and $Q$ exist.  First, we construct an interpolation. Let $Z$ be such that $M\itl{\eps}Z$ with the interleaving maps $(\xi,\eta)$ and $S\itl{\eps}Z$ with the interleaving maps $(\zeta,\nu)$. For the construction of the interpolated module, see ~\cite{ChaCohGli2009}. 
Now we set
\[
P = Z(\eps)
\]
the shifted version of $Z$. Then let
\[
f : M \rightarrow Z(\eps) \qquad \mbox{and} \qquad g : S\rightarrow Z(\eps)
\]
where $f$ and $g$ are the interleaving maps $\xi$ and $\zeta$ respectively. Note that as morphisms into $Z(\eps)$, $f$ and $g$ are 0-morphisms, that is, they are ungraded morphisms. Setting 
\[
N = \im f \qquad\text{and}\qquad
 Q = \im g
\]
we have the following set of short exact sequences:
\begin{center}
\begin{tikzcd}[column sep=small,row sep=tiny]
0  \arrow{r}&\ker f \arrow{r}&M \arrow{r}&\im f \arrow{r}& 0   \\
  0 \arrow{r}& \im f\arrow{r}& Z(\eps)\arrow{r} & \coker f \arrow{r}& 0 \\
  0 \arrow{r} &\im g\arrow{r}& Z(\eps) \arrow{r}& \coker g\arrow{r} & 0  \\
  0 \arrow{r} &\ker g \arrow{r}&S \arrow{r}& \im g \arrow{r} &0   
\end{tikzcd}
\end{center}
We can directly verify that $\ker f$, $\coker f$, $\coker g$ and $\ker g$ are $\eps$-interleaved with 0, hence completing the proof.
In the other direction assume $N,P,$ and $Q$ exist. This gives rise to the following short exact sequences, where the $\eps$ denote (possibly distinct) modules $\eps$-interleaved with $0$.
\begin{center}
\begin{tikzcd}
 0 \arrow{r} &\eps \arrow{r} &M \arrow[r, "i"]\arrow[d,"j \circ i"] &N \arrow{r} &0 \\
 0 \arrow{r} &N \arrow[r,"j"]& P \arrow{r} \arrow[d, "\cong"]& \eps \arrow{r} &0 \\
0 \arrow{r} &Q\arrow[r, "k"] &P\arrow{r}\arrow{u}& \eps \arrow{r} &0\\
0 \arrow{r}& \eps \arrow{r}& S \arrow[r,"\ell"]\arrow[u," k\circ\ell"] &Q \arrow{r}& 0
\end{tikzcd}
\end{center}
If we consider the composition of the first two exact sequences and the last two, we are in the fourth case of Proposition \ref{doesnotaddup}. This implies that there exists a $2\eps$-morphism $\varphi: P\rightarrow M$ and $\psi: P\rightarrow S$, each of which is the component of an appropriate interleaving. Hence, we can consider the following commutative diagram:
\begin{center}
\begin{tikzpicture}[scale=2]
\node (MT) at (-1,1) {$M$};
\node (NT) at (0,1) {$P$};
\node (PT) at (1,1) {$S$};
\node (M) at (-1,0) {$M$};
\node (N) at (0,0) {$P$};
\node (P) at (1,0) {$S$};
\node (MB) at (-1,-1) {$M$};
\node (NB) at (0,-1) {$P$};
\node (PB) at (1,-1) {$S$};
\path[->]
(MT) edge node[above]{$j\circ i$} (NT)
(PT) edge node[above]{$k\circ\ell$} (NT)
(M) edge node[below]{$j\circ i$} (N)
(P) edge node[below]{$k\circ\ell$} (N)
(MB) edge node[below]{$j\circ i$} (NB)
(PB) edge node[below]{$k\circ\ell$} (NB)
(MT) edge node[left]{$t^{2\eps}$} (M)
(NT) edge node[right]{$t^{2\eps}$} (N)
(PT) edge node[right]{$t^{2\eps}$} (P)
(NT) edge node[below]{$\varphi$} (M)
(NT) edge node[below]{$\psi$} (P)
(M) edge node[left]{$t^{2\eps}$} (MB)
(N) edge node[right]{$t^{2\eps}$} (NB)
(P) edge node[right]{$t^{2\eps}$} (PB)
(N) edge node[below]{$\varphi$} (MB)
(N) edge node[below]{$\psi$} (PB);
\end{tikzpicture}
\end{center}
This diagram commutes, since $t^{2\eps} = \varphi \circ j\circ i$ and $t^{2\eps} = \psi \circ k \circ \ell$ by the remark in the proof of the first case of Proposition \ref{doesnotaddup}. Hence, we have the required $2\eps$-interleaving given by $(\psi \circ j\circ i, \varphi\circ k \circ \ell)$.
\end{proof}
This decomposition helps give an interpretation to right and left interleaving in the case where the barcode exists.  The first short exact sequence is a right interleaving which shortens bars by changing the death time of a bar; the second sequence is a left interleaving, which lengthens the bars by changing the birth time of a bar; the third sequence is again a left interleaving which now shortens the bars by changing the birth time; finally the last sequence is a right interleaving which lengths bars by changing the death time. This interpretation of shortening and lengthening bars leads us to the following conjecture.

\begin{conjecture}
Any composition of suitable left and right interleavings yields an equivalence with interleavings.
\end{conjecture}
Essentially, we should be able to shorten and lengthen bars (when these notions are well defined) in any order, rather than just the order we list in Proposition~\ref{prop:equivalence}. Note since we use the fourth case of Proposition \ref{doesnotaddup}, the results do not hold for general persistence modules, but rather require projective dimension one. We believe this approach may help highlight what results hold for more general modules.

We conclude this section with the analysis of a special case: when a module is $\eps$-interleaved with the trivial module. This was studied extensively in ~\cite{chacholski2015multidimensional} for more complicated modules. Unfortunately, the results were not applicable directly, however the connection of left and right interleavings with ~\cite{chacholski2015multidimensional}  remains open. We conclude with a lemma that further illustrates  that interleaving with the trivial module has special structure. 
\begin{lemma}\label{lem:zerointerleaving}
If a module is $\eps$-interleaved with $0$, then it is both right and left $2\eps$-interleaved with $0$.
\end{lemma}
\begin{proof}
To prove the result, we consider the following short exact sequences illustrating left and right interleaving respectively: 
\[
0 \rightarrow 0 \xrightarrow{t^{2\eps}} A \xrightarrow{\cong} \coker(t^{2\eps}) \rightarrow 0
\]
\[
0 \rightarrow \ker(t^{2\eps}) \xrightarrow{\cong}  A \xrightarrow{t^{2\eps}}  0 \rightarrow 0
\]
It follows directly that $\ker(t^{2\eps})$ and $\coker(t^{2\eps})$ are $\eps$-interleaved with 0, fufilling Definitions~\ref{def:rinterleave} and \ref{def:linterleave} and hence $A$ is both $2\eps$-left and right interleaved with $0$.
\end{proof}

\section{Approximating Higher Pages}\label{sec:approx}

The main work in the proof is to track the approximation factors through the spectral sequence. Let $E$ be the Mayer-Vietoris spectral sequence associated to $(\Xs,\XCov)$. In the acyclic case, as in the case for many spectral sequences, the sequence collapses on the second page. Furthermore, the special structure of the second page, i.e. $E^2_{p,q}=0$ for $q>0$, eliminates the possibility of extension problems.  This allows for the homology of the space to be read off from the bottom row, and hence corresponding with the homology of the nerve (Theorem~\ref{thm:acyclic}). The extension problems which arise in our setting are further discussed in Section~\ref{sec:linking}. 

Therefore, a natural first step is to  compare the bottom row of the $E^2$ page with the homology of the nerve. 

\begin{proposition}
If $\XCov$ is an $\eps$-acyclic cover of $\Xs$, $(E^1_{*,0},d^1_{*,0})$ and $(\PC_*(\Nrv),\partial)$ are $2\eps$-interleaved as chain complexes.
\end{proposition}

\begin{proof}
The interleaving maps $\phi_p:E^1_{p,0}\to\PC_p(\Nrv)$ and $\psi_p:\PC_p(\Nrv)\to E^1_{p,0}$ are defined by the formulae
\[
\phi_p([v],I)=t^{\deg(v)-\deg(I)+2\eps}I\qquad\text{and}\quad\psi_p(I)=t^{2\eps}([v_I],I),
\]
where $v_I\in V$ is any vertex such that $\deg v_I = \deg I$. Note that the definition of $\psi$ requires a choice of $v_I$, but since $\XCov$ is an $\eps$-acyclic cover, $t^{2\eps}[v_I]$ is independent of this choice, so $\psi$ is well-defined.

A completely straightforward calculation now shows that $(\phi,\psi)$ is a $2\eps$-interleaving and that $\phi$ and $\psi$ commute with the differentials $\partial$ and $d^1$. (For the latter note that the differentials only really act on the information coming from the nerve, i.e. $I$, while the interleaving maps preserve this information.)
\end{proof}

Using Lemma \ref{induced}, this immediately yields:

\begin{corollary}\label{secondpage}
If $\XCov$ is an $\eps$-acyclic cover of $\Xs$, then $E^2_{*,0}$ and $\PH_*(\Nrv)$ are $2\eps$-interleaved as graded modules.
\end{corollary}
Note that setting $\eps=0$, recovers Theorem \ref{thm:acyclic}. We now observe that in the nerve construction, the dimension of the nerve is $D=\dim\Nrv$, all $(D+1)$-intersections are empty and hence 0. In this case, Corollary \ref{secondpage} can be sharpened:
\begin{remark}\label{rem:outsidebox}
For $d\geq D+1$, $E^2_{d,0}$ and $\PH_d(\Nrv)$ are both trivial and hence isomorphic.
\end{remark}

The next step is to establish a relation between $E^2$ and $E^\infty$.

\begin{proposition}\label{epsilon}
If $\XCov$ is an $\eps$-acyclic cover of $\Xs$, then $E^r_{p,q}\itl{\eps}0$ holds for all $p\in\ZZ$ and $q\neq 0$ and all $r\geq 1$.
\end{proposition}

\begin{proof}
Using 
\[
E^1_{p,q}=\bigoplus_{|I|=p+1}\PH_q(\XCv_I)
\]
(see Equation \eqref{E1}) and the definition of $\eps$-acyclic cover, we obtain the claim for $E^1_{p,q}$ with $q>0$. Since all the $E^r_{p,q}$ with $r>1$ are subquotients of $E^1_{p,q}$, the claim is now a direct consequence of Corollary \ref{subquotient}.
\end{proof}

We can now prove the following proposition:
\begin{proposition}\label{higherpages}
If $\XCov$ is an $\eps$-acyclic cover of $\Xs$, then $E^{r+1}_{*,0}\Litl{2\eps}E^r_{*,0}$ as graded modules for all $r\geq 2$.
\end{proposition}

\begin{proof}
Notice that $E^{r+1}_{p,0}=\ker d^r_{p,0}$, since the domain of $d^r_{p+r,-r+1}$ is $0$. We conclude that $\ker d^r_{p,0}\Litl{2\eps}E^r_{p,0}$ is true by Definition \ref{def:linterleave}, since $E^r_{p-r,r-1}\itl{\eps}0$ by Proposition \ref{epsilon}.
\end{proof}


If the spectral sequence collapses after finitely many steps, $E^2$ may already give a good approximation to $E^\infty$. This happens, for instance, if dimension of the nerve or underlying space are finite.  We define $D:=\dim\Nrv$, the maximum dimension of any simplex in $\Nrv$. Since simplices in $\Nrv$ correspond to non-empty intersections of cover elements, $D$ is also the smallest number such that any intersection of more than $D+1$ distinct cover elements is empty. Note that in the following the number of pages required until the spectral sequence collapses may be bounded by the dimension of the underlying space.

\begin{theorem}\label{infinitypage}
If $\XCov$ is an $\eps$-acyclic cover of $\Xs$ and $0<D<\infty$, then $E^\infty_{*,0}\Litl{2(D-1)\eps}E^2_{*,0}$ as graded modules. For $D=0,1$ we have $E^\infty_{*,0}\cong E^2_{*,0}$.
\end{theorem}

\begin{proof}
Since the intersections of more than $D+1$ cover elements are necessarily empty, $E^r_{p,q}=0$ holds for all $p>D$. Therefore, for $r>D$, we have $d^r=0$, since either the domain or codomain of each $d^r_{p,q}$ is zero. This immediately implies that the spectral sequence has collapsed by the $(D+1)$-th page, i.e. $E^{D+1}=E^{D+2}=\ldots$. This concludes the proof for $D=0$. For $D>0$, using Proposition \ref{higherpages}, this shows that
\[
E^{\infty}_{*,0}=E^{D+1}_{*,0}\Litl{2\eps}E^D_{*,0}\Litl{2\eps}\ldots\Litl{2\eps}E^3_{*,0}\Litl{2\eps}E^2_{*,0}
\]
and therefore $E^{\infty}_{*,0}\Litl{2(D-1)\eps}E^2_{*,0}$ by the triangle inequality for left interleavings.
\end{proof}

\begin{remark}\label{specialcase}
For dimension $d>D$, since all the modules are trivial it follows that $E^\infty_{d,0}\cong\PH_d(\Nrv)$.
\end{remark}

A similar argument shows a weaker property without any assumptions on the dimension of the nerve.

\begin{theorem}
If $\XCov$ is an $\eps$-acyclic cover of $\Xs$ and $n>0$, we have $E^\infty_{n,0}\Litl{2(n-1)\eps}E^2_{n,0}$. For $n=0$ we have $E^\infty_{n,0}\cong E^2_{n,0}$.
\end{theorem}

\begin{proof}
Observe that for $r>n>0$, we have $d^r_{n,0}=0$ and $d^r_{n+r,-r+1}=0$, since $E^r_{n-r,r-1}$ and $E^r_{n+r,-r+1}$ are zero. Therefore, $E^{n+1}_{n,0}=E^{n+2}_{n,0}=\ldots$. Combined with Proposition \ref{higherpages} this shows that
\[
E^{\infty}_{n,0}=E^{n+1}_{n,0}\Litl{2\eps}E^n_{n,0}\Litl{2\eps}\ldots\Litl{2\eps}E^3_{n,0}\Litl{2\eps}E^2_{n,0}
\]
and therefore $E^\infty_{n,0}\Litl{2(n-1)\eps}E^2_{n,0}$ by the triangle inequality for left interleavings.

The case $n=0$ holds since for $r>1$ all differentials to and from $E^r_{0,0}$ are zero.
\end{proof}

\section{From $E^\infty$ to Homology}\label{sec:linking}

If there were no extension problems, the direct sum of the antidiagonals on the $E^\infty$ page of the spectral sequence would be isomorphic to the homology of the space, and completing the proof would be straightforward. However, when dealing with persistence modules, we \textbf{do} have to worry about extensions. As noted before, in the acyclic case, $E^2_{p,q}=0$ for all $q>0$, so the only possible extension is the trivial one. If we replace the $\eps$-modules below by 0, we see that each step becomes an isomorphism.  
We now show how to infer an Approximate Nerve Theorem from these results. For technical reasons, we have to distiguish between several cases depending on the dimension of the nerve and beyond that dimension.

\begin{proposition}\label{linking}
If $\XCov$ is an $\eps$-acyclic cover of $\Xs$ and $D<\infty$, $\PH_d(\Xs)\Ritl{2d\eps}E^\infty_{d,0}$ holds for $0\leq d \leq D$.
\end{proposition}

\begin{proof}
By Theorem \ref{convergence}, we already know that $E$ converges to $\PH_*(\Xs)$. Explicitly, this means that a filtration $(\PH_*(\Xs)^p)_{p\in\ZZ}$ is defined on $\PH_*(\Xs)$ such that
\[
E^\infty_{p,q}\cong\frac{\PH_{p+q}(\Xs)^p}{\PH_{p+q}(\Xs)^{p-1}}.
\]
In the process of reconstructing $\PH_n(\Xs)=\PH_n(\Xs)^n$ from $E^\infty_{p,q}$ with $p+q=n$, we therefore encounter a series of extension problems. The effect of each of these extension problems in our case, however, is simply to add an error of $2\eps$ to our approximation of $\PH_n(\Xs)$. Specifically, we have
\begin{equation}\label{ddag}
\frac{\PH_n(\Xs)^n}{\PH_n(\Xs)^{p-1}}\Ritl{2\eps}\frac{\PH_n(\Xs)^n}{\PH_n(\Xs)^{p}}
\end{equation}
for each $p\neq n$ (equivalently $q\neq 0$). To see this, observe that the sequence
\begin{equation}\label{ssddag}
0\to\frac{\PH_n(\Xs)^{p}}{\PH_n(\Xs)^{p-1}}\to\frac{\PH_n(\Xs)^n}{\PH_n(\Xs)^{p-1}}\to\frac{\PH_n(\Xs)^n}{\PH_n(\Xs)^{p}}\to 0
\end{equation}
is exact and
\begin{equation}\label{epseq}
\frac{\PH_n(\Xs)^{p}}{\PH_n(\Xs)^{p-1}}=E^\infty_{p,q}\itl{\eps}0
\end{equation}
holds by Proposition \ref{epsilon} if $q\neq0$. Since the left most term is $\eps$-interleaved with 0, \eqref{ddag} then follows by Definition \ref{def:rinterleave}. The claim now follows inductively. For $0\leq n\leq D$, we have
\[
\PH_n(\Xs)\cong\frac{\PH_n(\Xs)^n}{\PH_n(\Xs)^{-1}} \Ritl{2\eps}\ldots\Ritl{2\eps}\frac{\PH_n(\Xs)^n}{\PH_n(\Xs)^{n-2}}\Ritl{2\eps}\frac{\PH_n(\Xs)^n}{\PH_n(\Xs)^{n-1}}\cong E^\infty_{n,0},
\]
Since $n\leq D$, there are at most $D$ $2\eps$-right interleavings, proving the result by Proposition \ref{prop:rightleft}. 
\end{proof}
Note that in the case where $\eps=0$, the extensions become trivial as the maps in the filtration are isomorphisms by exactness. 
The second case is for $\PH_d(\Xs)$ when $d>D$.

\begin{proposition}\label{linking_outside}
If $\XCov$ is an $\eps$-acyclic cover of $\Xs$ and $D<\infty$, $\PH_d(\Xs)\Ritl{2(D+1)\eps}E^\infty_{d,0}$ holds for $d > D$.
\end{proposition}
\begin{proof}
For $n>D$, we use the fact that $E^\infty_{p,q}\cong 0$ holds for all $p>D$ (equivalently $q<n-D$). The short exact sequence \eqref{ssddag} for these $p$ implies that 
\[
\frac{\PH_n(\Xs)^{n}}{\PH_n(\Xs)^{D}} \cong\frac{\PH_n(\Xs)^{n}}{\PH_n(\Xs)^{D+1}} \cong\ldots\cong \frac{\PH_n(\Xs)^n}{\PH_n(\Xs)^{n-1}}.
\]
Using \eqref{ssddag} and \eqref{epseq} we obtain the following sequence of right interleavings
\[
\PH_n(\Xs) \cong \frac{\PH_n(\Xs)^n}{\PH_n(\Xs)^{-1}} \Ritl{2\eps}\ldots\Ritl{2\eps} \frac{\PH_n(\Xs)^n}{\PH_n(\Xs)^{D-1}}\Ritl{2\eps} \frac{\PH_n(\Xs)^n}{\PH_n(\Xs)^{D}} \cong \frac{\PH_n(\Xs)^n}{\PH_n(\Xs)^{n-1}} \cong  E^\infty_{n,0}.
\]
By counting that there are $(D+1)$ $2\eps$-right interleavings, we obtain the result. 
\end{proof}

For completeness we add one further case: where the dimension of the space is lower than the dimension of the nerve. 
For example, the nerve of a cubical cover of $k$-dimensional Euclidean space has $D=2^k$. We could redo much of our work for cubical complexes, however the following result shows this is unnecessary. Let $\spaceD := \dim\Xs$. For the case, $D>\spaceD$ we show the approximation constant depends on $\spaceD$ instead of $D$. 
\begin{proposition}\label{linking_lowdimspace}
If $\XCov$ is an $\eps$-acyclic cover of $\Xs$ and $\spaceD<\infty$, $\PH_d(\Xs)\Ritl{2\spaceD\eps}E^\infty_{d,0}$ holds for all $d$.
\end{proposition}
\begin{proof}
The proof follows as in the above propositions. However, since $\spaceD$ is the dimension of the space
\[
\frac{\PH_n(\Xs)^{p}}{\PH_n(\Xs)^{p-1}} = 0, \qquad p\leq n-\spaceD-1,
\]
Therefore using \eqref{ssddag} and \eqref{epseq} we obtain the sequence
\[
 \PH_n(\Xs) \cong  \frac{\PH_n(\Xs)^n}{\PH_n(\Xs)^{-1}}
\cong  \frac{\PH_n(\Xs)^n}{\PH_n(\Xs)^{n-\spaceD-1}} \Ritl{2\eps} \ldots \Ritl{2\eps} \frac{\PH_n(\Xs)^n}{\PH_n(\Xs)^{n-2}} \Ritl{2\eps} \frac{\PH_n(\Xs)^n}{\PH_n(\Xs)^{n-1}}
 \cong  E^\infty_{n,0}.
\]
There are $\spaceD$ $2\eps$-right interleavings, proving the result. 
%
%
\end{proof}



\section{Main Theorems}\label{sec:thms}

Here we connect the results of the previous two sections to obtain our main result. The idea is to consider the chain of approximations. Unfortunately there are several cases we have to consider depending on the dimension of the nerve and the space. The basic idea however is to consider the relationships in the sequence 
\[
\PH_*(\Nrv)\itl{} E^2_{*,0}\itl{}E^\infty_{*,0}\itl{}\PH_*(X).
\]
where we recall that $X$ is a filtered simplicial complexes and $\Nrv$ is another filtered complex given by the nerve of a cover on $X$.
Before stating the result with the tight constant, we consider an easy case of the result which does not use the specific properties of left and right interleavings. Recall that a $2\eps$-left or right interleaving implies a $2\eps$-interleaving. 

\begin{theorem}\label{thm:easy}
If $\XCov$ is an $\eps$-acyclic cover of $\Xs$ and $D<\infty$, we have $\PH_*(\Xs)\itl{(4D+2)\eps}\PH_*(\Nrv)$.
\end{theorem}

\begin{proof}
Assuming $D>0$ and composing interleavings with constants, we obtain
\[
\PH_*(\Nrv)\itl{2\eps} E^2_{*,0}\itl{2(D-1)\eps}E^\infty_{*,0}\itl{2(D+1)\eps}\PH_*(\Xs).
\]
The first interleaving is from Corollary \ref{secondpage} and the second follows from Theorem \ref{infinitypage}. Finally the last interleaving
follows from  Proposition \ref{linking} for $0\leq d\leq D$ and Proposition \ref{linking_outside} for $d>D$. Adding the terms we obtain the result.
The case $D=0$ is straightforward.
\end{proof}

\begin{theorem}\label{thm:main}
Let $\minD=\min(D,\spaceD)$. If $\XCov$ is an $\eps$-acyclic cover of $\Xs$ and $Q<\infty$, we have $\PH_*(\Xs)\itl{2(\minD+1)\eps}\PH_*(\Nrv)$.	
\end{theorem}

\begin{proof}
Observe that in the proof of the previous theorem,
For $0\leq d \leq D $ and $\spaceD\geq D$,
 the precise relationship is
\[
\PH_d(\Xs)\Ritl{2D\eps}E^\infty_{d,0}\Litl{2D\eps}E^2_{d,0}\itl{2\eps}\PH_d(\Nrv).
\]
The first interleaving follows from Proposition \ref{linking}, the second from Theorem \ref{infinitypage} and the last one from Corollary \ref{secondpage}. However, the interleaving obtained from Theorem \ref{infinitypage} is a left interleaving, whereas the one from Proposition \ref{linking} is a right interleaving. By Proposition \ref{doesnotaddup}, together these imply
\[
\PH_d(\Nrv)\itl{2\eps}E^2_{d,0}\itl{2D\eps}\PH_d(\Xs).
\]
For $d > D $ and $\spaceD\geq D$, 
\[
\PH_d(\Xs)\Ritl{2(D+1)\eps}E^\infty_{d,0}\cong\PH_d(\Nrv),
\]
where the isomorphism follows from Remark~\ref{specialcase} and the interleaving follows from Proposition \ref{linking_outside}. As a right interleaving implies interleaving, this proves this case.
Finally, for $\spaceD<D$, we note the spectral sequence stabilizes after $\spaceD+1$ steps, therefore the relationship is
\[
\PH_*(\Xs)\Ritl{2\spaceD \eps}E^\infty_{*,0}\Litl{2(\spaceD-1) \eps}E^2_{*,0}\itl{2\eps}\PH_*(\Nrv),
\]
where the right interleaving is due to Proposition~\ref{linking_lowdimspace}. Again noting that right and left interleavings do not interact, we obtain
\[
\PH_*(\Nrv)\itl{2\eps}E^2_{*,0}\itl{2\spaceD\eps}\PH_*(\Xs).
\]
We can now directly verify that the approximation is bounded by $2(\min(D,\spaceD)+1)\eps$, concluding the proof. 
\end{proof}


Using an  analogous argument without any assumptions on $D$ or $\spaceD$, we obtain
\begin{theorem}
If $\XCov$ is an $\eps$-acyclic cover of $\Xs$, $\PH_n(\Xs)\itl{2(n+1)\eps}\PH_n(\Nrv)$. 
\end{theorem}
\begin{proof}
The key observation is that since we have a first quadrant spectral sequence,  $E^{n+1}_{p,q} \cong E^{\infty}_{p,q}$ for $0\leq p+q \leq n$.
Applying Propositions~\ref{linking} and \ref{linking_outside}, yields
\[
\PH_n(\Xs)\Ritl{2n\eps}E^\infty_{n,0}\cong E^{n+2}_{n,0}\Litl{2n\eps}E^2_{n,0}\itl{2\eps}\PH_n(\Nrv).
\]
As in Theorem \ref{thm:main}, combining the interleavings yields the result. 
\end{proof}

\section{Applications}\label{sec:example}

We prove a simple result of a possible application of our main result. While the result is not new, the proof is an immediate consequence of our result. There are many related approximatation results in the literature (for example, ~\cite{bobrowski2015topology,chazal2009analysis,niyogi2008finding,chazal2009sampling,chazal2008towards,sheehy2012multicover}). We do not provide a comprehensive account of these approximation results but provide two example applications to illustrate the Approximate Nerve Theorem.  

Throughout this section we use the function $g$ on the nerve which was defined in Section~\ref{sec:epscover}, which inserts a simplex into nerve as soon as the corresponding intersection is non-empty.

\begin{theorem}
Given a $c$-Lipschitz function $f$ on a $D$-dimensional manifold $X$ embedded in Euclidean space with positive reach $\rho$,  given an $\eps$-sample of the space with $\eps<\rho$, consider the cover of balls of radius $\eps$ centered at the sample points. Let $h:\Nrv\rightarrow \RR$ be the function defined by the formula
$$ h(I) = \max\limits_{i\in I} f(x_i) $$
where $x_i$ is the corresponding sample point. Then,  
\[
d_I(\PH_*(X,f),\PH_*(\Nrv,h))\leq (4D+3)c\eps.
\]
\end{theorem}
\begin{proof}
Ignoring the function for the time being, since we have  an $\eps$-sample, balls of radius $\eps$ centered at the sample points form a cover of the manifold $X$.  Since the reach is larger than $\eps$, it follows that the union of balls form a good cover of $X$. Now, we show that for a $c$-Lipschitz function, this is a $2c\eps$-acyclic cover. Using the construction in Section~\ref{sec:epscover}, we note that the maximum value attained in any cover element is 
\[
f(\XCv_I) \leq g(I)+2c\eps.
\] 
where $g$ is defined in Equation~\ref{eq:definitionofg}.
Hence, after $2c\eps$, the sublevel set fills the entire cover element, so it is a $2c\eps$-good cover. In this construction, we also note that if the elements are $2c\eps$-interleaved with the trivial diagram so are all intersections. This gives an approximation of $(2D+1)2c\eps$. Finally, we 
note that since the cover elements are bounded in size by $\eps$ and the definition of $g$, $|h-g|\leq c\eps$. Adding these constants together yields an interleaving which implies the result. 
\end{proof}
The bound above is not meant to be tight as a slightly longer argument would remove a $c\eps$, and many similar results have been proven. Importantly it illustrates that we can appoximate the sublevel set persistence with a single filtration rather than an image between two cover elements as in \cite{bobrowski2015topology}  without requiring any one sublevel set to have a good cover.  We do note that in this instance, it is possible but cumbersome to construct an explicit functional interleaving. An almost identical result can also be stated replacing reach with other measures such as convexity radius, homotopy feature size, etc. 

We also wish to derive an Approximate Nerve Theorem for $\eps$-acyclic covers of triangulable spaces directly from the one for simplicial complexes. However, covers of triangulable spaces by triangulable subsets are too general for this, as their triangulations may not interact well. To circumvent this issue, we introduce the following technical notion.

\begin{definition}
Suppose $\overline\YCov=(\YCv_i)_{i\in\Lambda}$ is a cover of a locally compact triangulable space $\Ys$. We say $\overline\YCov$ is a {\em triangulable cover} if there exists some triangulation $(\widetilde\Xs,h)$ of $\Ys$ such that each cover element $\YCv_i$ is the image of a subcomplex of $\widetilde\Xs$ under $h$.
\end{definition}

Such covers are very common in practical applications. The notion of $\eps$-acyclic cover is analogous to the one for simplicial complexes, however, continuous persistence modules must be used. A triangulable cover by itself is not filtered, but we will impose a filtration on it by specifying a function on each cover element. We do {\em not} require that the triangulable cover condition holds at the intermediate stages of the filtration.

First we prove a preliminary Lemma to establish that a filtered cover of a triangulable space can be approximated arbitrarily well by one whose filtration is given by piecewise linear functions.

\begin{lemma}\label{triangulation}
Let $\Ys$ be a locally compact triangulable space and $\overline\YCov=(\YCv_i)_{i\in\Lambda}$ a locally finite cover of $\Ys$. Suppose $\overline\YCov$ is triangulable, with triangulation $(\widetilde\Xs,h)$. Let $\eps>0$. Given continuous functions $f:\Ys\to\RR$ and $f_i:\YCv_i\to\RR$, $i\in\Lambda$, there exists a subdivision $\Xs$ of $\widetilde\Xs$ such that for each simplex $\sigma$ of $\Xs$ we have
\[
\max_{x\in|\sigma|} f(h(x)) - \min_{x\in|\sigma|} f(h(x))<\eps\qquad\text{and}\qquad\max_{x\in|\sigma|} f_i(h(x)) - \min_{x\in|\sigma|} f_i(h(x))<\eps\quad\text{for all $i\in\Lambda$.}
\]
\end{lemma}

\begin{proof}
Let $(\widetilde\Xs,h)$ be a triangulation of $\Ys$. By local compactness, $\widetilde\Xs$ is locally finite, so $|\widetilde\Xs|$ is metrizable. Choose a metric $d$ on $|\widetilde\Xs|$. Since $fh$ is uniformly continuous on each simplex $\tilde\sigma$, there exists a $\delta(\tilde\sigma)>0$ such that $d(x_1,x_2)<\delta(\tilde\sigma)$ implies $|f(h(x_1))-f(h(x_2))|<\eps$ for all $x_1,x_2\in|\tilde\sigma|$. Since $f_ih$, $i\in\Lambda$, is uniformly continuous on each simplex $\tilde\sigma$, there exists a $\delta_i(\tilde\sigma)>0$ such that $d(x_1,x_2)<\delta_i(\tilde\sigma)$ implies $|f_i(h(x_1))-f_i(h(x_2))|<\eps$ for all $x_1,x_2\in|\tilde\sigma|$. Since the cover $\overline\YCov$ is locally finite, each simplex $\tilde\sigma\in\widetilde\Xs$ is only contained in finitely many cover elements $\YCv_{i_1},\ldots,\YCv_{i_k}$. Let $\delta'(\tilde\sigma)=\min\{\delta(\tilde\sigma),\delta_{i_1}(\tilde\sigma),\ldots,\delta_{i_k}(\tilde\sigma)\}$. Using iterated barycentric subdivision on each simplex $\tilde\sigma$, we can now construct a subdivision $\Xs$ of $\widetilde\Xs$ such that the diameter of each simplex in $|\tilde\sigma|$ is less than $\delta'(\tilde\sigma)$ and so $\Xs$ has the desired property.
\end{proof}

\begin{corollary}\label{plaprox}
Under the assumptions of Lemma \ref{triangulation}, the piecewise linear functions $\hat f:|\Xs|\to\RR$ and $\hat{f_i}:|\XCv_i|\to\RR$ defined on the vertices by $\hat f(v)=f(h(v))$ and $\hat{f_i}(v)=f_i(h(v))$ and extended affinely over the simplices satisfy $\|\hat f-fh\|_{\infty}\leq\eps$ and $\|\hat{f_i}-f_ih\|_{\infty}\leq\eps$, respectively. Consequently, $\|\min_{i\in\Lambda}\hat{f_i}-\hat f\|\leq2\eps$.
\end{corollary}

The final inequality means that upon replacing the functions $f$ and $f_i$ by piecewise linear approximations, the compatibility condition $f=\min_{i\in\Lambda}f_i$ remains approximately true. This is important, because the compatibility condition is needed to invoke the Approximate Nerve Theorem for filtered simplicial complexes. We now have the necessary tools to prove an Approximate Nerve Theorem for triangulable spaces.

\begin{proposition}
Let $\Ys$ be a locally compact triangulable space and $\overline\YCov=(\YCv_i)_{i\in\Lambda}$ a locally finite triangulable cover of $\Ys$. Let $f:\Ys\to\RR$ and $f_i:\YCv_i\to\RR$, $i\in\Lambda$, be continuous functions such that $f=\min_{i\in\Lambda}f_i$. Let $\Nrv(\YCov)=(\Nrv,g)$ be the nerve of the filtered cover $\YCov=(\YCv_i,f_i)_{i\in\Lambda}$ of $(\Ys,f)$. Let $D=\dim\Nrv$, $\spaceD=\dim\Ys$ and $Q=\min(D,\spaceD)<\infty$. If $\YCov$ is $\eps$-acyclic, $\PH_*(\Ys,f)\itl{2(Q+1)\eps+\eta}\PH_*(\Nrv, g)$ holds for any $\eta>0$. In particular, 
\[
d_I(\PH_*(\Ys,f),\PH_*(\Nrv, g))\leq 2(Q+1)\eps.
\]
\end{proposition}

\begin{proof}
Let $(\widetilde\Xs,h)$ be the triangulation from the definition of triangulable cover. By Lemma \ref{triangulation} and its Corollary, there is a subdivision $\Xs$ of $\widetilde\Xs$ and a corresponding cover $\overline\XCov=(\XCv_i)_{i\in\Lambda}$ of $\Xs$, satisfying $h(|\XCv_i|)=\YCv_i$, such that the piecewise linear functions $\hat f:|\Xs|\to\RR$ associated to $fh$ and $\hat{f_i}:|\Xs|\to\RR$ associated to $f_ih$ satisfy $\|\hat f-fh\|_\infty<\delta$ and $\|\hat{f_i}-f_ih\|_\infty<\delta$, where $\delta>0$ is to be chosen later.

Recall from Section~\ref{sec:prelim} that there are two functors: the (natural) restriction functor $I_\delta:\Vect^{(\RR,\leq)}\to\Vect^{(\delta\ZZ,\leq)}$ given by $I_\delta(F)=Fi_\delta$ and an extension functor $P_\delta:\Vect^{(\delta\ZZ,\leq)}\to\Vect^{(\RR,\leq)}$ given by $P_\delta(F)=Fp_\delta$. Next, observe that defining $u^\delta(x):=\lceil\frac{u(x)}{\delta}\rceil\delta$, whenever $u$ is a real-valued function, we have
\begin{equation}\label{fdelta}
P_\delta(I_\delta(\PH_*(\XCv_I,\hat f_I)))=\PH_*(\XCv_I,\hat f_I^\delta)\qquad\text{and}\qquad P_\delta(I_\delta(\PH_*(\Xs,\hat f)))=\PH_*(\Xs,\hat f^\delta).
\end{equation}
Using the interleavings/isomorphisms provided by Proposition \ref{functionalcase}, Proposition \ref{filtrations}, Proposition \ref{rightinverse} and equation \eqref{fdelta} we obtain in turn
\[
\PH_*(\YCv_I,f_I)\itl{\delta}\PH_*(\YCv_I,\hat f_Ih^{-1})\cong\PH_*(\XCv_I,\hat f_I)\itl{\delta}P_\delta(I_\delta(\PH_*(\XCv_I,\hat f_I)))=\PH_*(\XCv_I,\hat f_I^\delta).
\]
By the same logic and using Corollary \ref{plaprox} to obtain the additional $2\delta$-interleaving in the middle, we have\footnote{Note that $(\min_{i\in\Lambda}\hat{f_i})^\delta=\min_{i\in\Lambda}(\hat f_i^\delta)$, so we can drop the parentheses in the final expression.}
\[
\PH_*(\Ys,f)\itl{\delta}\PH_*(\Ys,\hat fh^{-1})\cong\PH_*(\Xs,\hat f)\itl{2\delta}\PH_*(\Xs,\min_{i\in\Lambda}\hat{f_i})\itl{\delta}P_\delta(I_\delta(\PH_*(\Xs,\min_{i\in\Lambda}\hat{f_i})))=\PH_*(\Xs,\min_{i\in\Lambda}\hat f_i^\delta).
\]
Since $\YCov$ is $\eps$-acyclic and $\PH_*(\YCv_I,f_I)\itl{2\delta}\PH_*(\XCv_I,\hat f_I^\delta)$ for all $I$, $\XCov$ is a $(\eps+2\delta)$-acyclic cover of $(\Xs,\min_{i\in\Lambda}\hat f_i^\delta)$. In fact, it is $(p_\delta(\eps)+2\delta)$-acyclic. To see this, note that the $\RR$-persistence modules $\PH_*(\XCv_I,\hat f_I^\delta)$ and $\PH_*(\Xs,\hat f^\delta)$ may be represented as $\delta\ZZ$-persistence modules, since their filtrations only change at $\delta\ZZ$ (see discussion following Proposition \ref{discreteitl}). This means that they lie in the image of the isometry $P_\delta$. In particular, interleaving distances between such modules must be multiples of $\delta$. Therefore, Theorem \ref{thm:main} applies to the pair $(\Xs,\XCov)$, where $\XCov=(\XCv_i,\hat f_i^\delta)_{i\in\Lambda}$. Taking into account that $\Nrv(\overline\XCov)=\Nrv(\overline\YCov)=\Nrv$, this means that
\[
\PH_*(\Xs,\min_{i\in\Lambda}\hat f_i^\delta)\itl{2(Q+1)(p_\delta(\eps)+2\delta)}\PH_*(\Nrv,g_\delta),
\]
where $g_\delta$ is the function on the nerve corresponding to the family of filtrations $(\hat f_i^\delta)_{i\in\Lambda}$. Using Proposition \ref{isometry} we may now once again regard these as $\RR$-persistence modules. It remains to compare $g_\delta$ with the function $g$ corresponding to the family $(f_i)_{i\in\Lambda}$. Note that replacing each $f_ih$ by $\hat f_i^\delta$ changes the function values by at most $2\delta$, therefore we have $\|g-g_\delta\|_\infty\leq2\delta$. Using Remark \ref{functionalcase2} we conclude that $\PH_*(\Nrv,g_\delta)\itl{2\delta}\PH_*(\Nrv,g)$. Combining all these observations, we have
\[
\PH_*(\Ys,f)\itl{4\delta}\PH_*(\Xs,\min_{i\in\Lambda}\hat f_i^\delta)\itl{2(Q+1)(p_\delta(\eps)+2\delta)}\PH_*(\Nrv,g_\delta)\itl{2\delta}\PH_*(\Nrv,g),
\]
so $\PH_*(\Ys,f)$ and $\PH_*(\Nrv,g)$ are $(2(Q+1)\eps+(4Q+10)\delta)$-interleaved, using $p_\delta(\eps)\leq\eps$. Choosing $\delta:=\frac{\eta}{4Q+10}$ completes the proof.
\end{proof}

\section{Lower Bounds}\label{sec:lowerbnd}

Here we construct simple examples to show that the bounds in Corollary \ref{secondpage}, Theorem \ref{infinitypage} and Proposition \ref{linking} are sharp. For each example, we compute the homology of the nerve, the homology of the filtered simplicial complex and the $E^1,E^2$ and $E^\infty$ pages of the spectral sequence (up to isomorphism). For better readability, we use the notations
\[
[a,b]=\{k\in\ZZ\mid a\leq k\leq b\}\qquad\text{and}\qquad[a,b)=[a,b]\setminus\{b\}.
\]
Without loss of generality, we work with $\epsilon=1$, otherwise simply multiply each time in the filtration by $\epsilon$. To simplify the exposition, the pages of the spectral sequence are not computed directly, but rather inferred from the homology of the space and various intersections of its cover elements. 

In each of the two examples provided, the filtration is defined on the total space $\Xs$. The cover elements $\XCv_i$ are assumed to be equipped with the induced filtrations (see Remark \ref{inducedcov}). Each example illustrates the tightness of each step of our approximation proof. To construct a topological example which achieves all three, we can simply take the direct sum of the three examples. 

\subsection{First Example}

Our first example realizes the bounds in Corollary \ref{secondpage} and Theorem \ref{infinitypage}. Let $\Xs$ be the $D$-sphere, realized as the boundary of the $(D+1)$-simplex with vertex set $[0,D+1]$. A cover $\XCov$ of $\Xs$ is given by its set of maximal faces, i.e. $\XCov=\{\XCv_i\mid i\in[0,D+1]\}$, where $\XCv_i$ is the $D$-simplex spanned by $[0,D+1]\setminus\{i\}$.

We also define a filtration $\Xs^0\leq\Xs^2\leq\ldots\leq\Xs^{2D+2}$ by adding one cover element at a time, i.e.
\[
\Xs^{2j}=\XCv_0\cup\XCv_1\cup\ldots\cup\XCv_j.
\]

\begin{proposition}
The homology of the nerve of $\XCov$ is given by
\[
\PH_q(\Nrv)\cong\begin{cases}\kk[t];&q=0,\\
t^2\kk[t];&q=D,\\
0;&\text{otherwise.}
\end{cases}
\]
\end{proposition}

\begin{proof}
At time $0$, the vertices $1,\ldots,D+1$ are born in $\Xs$. For $I\subseteq[0,D+1]$ each $\XCv_I$ except $\XCv_{[1,D+1]}$ and $\XCv_{[0,D+1]}$ contains one of these vertices, so the nerve at time $0$ consists of all $I\subseteq[0,D+1]$, except for $[1,D+1]$ and $[0,D+1]$. At time $2$, the vertex $0$ is born, which corresponds to the birth of $[1,D+1]$ in the nerve. Since $\XCv_{[0,D+1]}$ is always empty, $\Nrv^j$ is contractible for $j=0,1$ and homeomorphic to a $D$-sphere for $j\geq 2$.
\end{proof}

\begin{proposition}
The homology of the filtered simplicial complex $\Xs$ is given by
\[
\PH_q(\Xs)\cong\begin{cases}\kk[t];&q=0,\\
t^{2D+2}\kk[t];&q=D,\\
0;&\text{otherwise.}
\end{cases}
\]
\end{proposition}

\begin{proof}
$\Xs^j$ is contractible at the times $j=0,\ldots,2D+1$. For $j\geq 2D+2$ it is homeomorphic to a $D$-sphere.
\end{proof}

Computing the $E^1$ page requires some preparation, namely simplifying $\XCv_I^{2j}$.

\begin{proposition}
Suppose that $\emptyset\neq I\subseteq[0,D+1]$ and let $j\in[0,D+1]$. If $j\geq\min I$, $\XCv_I^{2j}$ is a $(D+1-|I|)$-simplex. If $j<\min I$, $\XCv_I^{2j}$ is the join of a $(j-1)$-sphere, realized as the boundary of a $j$-simplex, and a $(D-|I|-j)$-simplex. We allow $D-|I|-j=-1$ and interpret ``$(-1)$-simplex'' as the empty set.
\end{proposition}

\begin{proof}
Observe that
\[
\XCv_I^{2j}=\XCv_I\cap\Xs^{2j}=\XCv_I\cap(\XCv_0\cup\ldots\cup\XCv_j)=\XCv_{I\cup\{0\}}\cup\XCv_{I\cup\{1\}}\cup\ldots\cup\XCv_{I\cup\{j\}}.
\]
For $j\geq\min I$, one of the terms is $\XCv_{I\cup\{\min I\}}=\XCv_I$, so $\XCv_I^{2j}=\XCv_I$ is the $(D+1-|I|)$-simplex with vertices $[0,D+1]\setminus I$. (Intersecting with $\XCv_i$ corresponds to removing the vertex $i$.)

For $j<\min I$, $\XCv_{I\cup\{k\}}$ (where $k\in[0,j]$) is the $(D-|I|)$-simplex with vertices $[0,D+1]\setminus(I\cup\{k\})$. So, $\XCv_I^{2j}$ is the complex spanned by all the simplices of the form $J\cup([j+1,D+1]\setminus I)$ where $J$ is a $j$-element subset of $[0,j]$. But this means precisely that $\XCv_I^{2j}$ is the simplicial join of the $(D-|I|-j)$-simplex $[j+1,D+1]\setminus I$ and the $(j-1)$-sphere, realized as the boundary of the $j$-simplex $[0,j]$.
\end{proof}

\begin{proposition}
Suppose that $\emptyset\neq I\subseteq[0,D+1]$. Then
\[
\PH_q(\XCv_I)\cong\begin{cases}\frac{t^{2q+2}\kk[t]}{t^{2q+4}\kk[t]};&q=D-|I|>0,I=[q+2,D+1],\\
\kk[t]\oplus\frac{t^2\kk[t]}{t^4\kk[t]};&q=0,I=[2,D+1],\\
\kk[t];&q=0<D-|I|\text{ or }q=0,|I|=D+1,I\neq[1,D+1],\\
t^2\kk[t];&q=0,I=[1,D+1],\\
0;&\text{otherwise.}
\end{cases}
\]
\end{proposition}

\begin{proof}
The join of a sphere and a non-empty simplex is contractible, so $\XCv_I^{2j}$ can only be non-acyclic is if it is the join of a sphere and an empty simplex. By the previous proposition this occurs precisely if $D=|I|+j-1$ and $I=[j+1,D+1]$ (the latter is required so that $\min I>j$) in which case $\XCv_I^{2j}$ is a $(j-1)$-sphere. If $j-1>0$, this means that $\Hg_{j-1}(\XCv_I^{2j})\cong\kk$ and $\Hg_0(\XCv_I^{2j})\cong\kk$, if $j-1=0$, it means that $\Hg_{j-1}(\XCv_I^{2j})\cong\kk^2$, and for $j=0$ all homology groups (corresponding to $\XCv_{[1,D+1]}$) are trivial. In all other cases, $\XCv_I^{2j}$ is contractible, so $\Hg_0(\XCv_I^{2j})\cong\kk$. The remaining homology groups are $0$. Setting $q=j-1$ completes the proof.
\end{proof}

Note that we have computed persistent homology slice-wise, i.e. by computing the simplicial homology at each step of the filtration. To infer the correct $\kk[t]$-module from this, we have used the facts that the filtration only changes at even times and that once it is born, the first class appearing in dimension $0$ lives forever. One immediate consequence of these computations is the following.

\begin{corollary}
The cover $\XCov$ is $1$-acyclic.
\end{corollary}

Since we already know that
\[
E^1_{p,q}=\bigoplus_{|I|=p+1}\PH_q(\XCv_I),
\]
the previous proposition also immediately yields the $E^1$ page.

\begin{corollary}
The $E^1$ page of the Mayer-Vietoris spectral sequence of $(\Xs,\XCov)$ is given by:
\begin{center}
  \begin{tikzpicture}[xscale=1.3,yscale=1.1]
	\node (E50) at (7,0) {$\kk[t]^{\binom{D+2}{D+1}-1}\oplus t^2\kk[t]$};
 	\node (E40) at (4.7,0) {$\kk[t]^{\binom{D+2}{D}}\oplus\frac{t^2\kk[t]}{t^4\kk[t]}$}; 
 	\node (E31) at (3,1) {$\frac{t^4\kk[t]}{t^6\kk[t]}$}; 
 	\node (E22) at (2,2) {$\ddots $};
	\node (E13) at (1,3) {$\frac{t^{2D-2}\kk[t]}{t^{2D}\kk[t]}$};
 	\node (E04) at (0,4) {$\frac{t^{2D}\kk[t]}{t^{2D+2}\kk[t]}$}; 
 	\node (E30) at (3,0) {$\kk[t]^{\binom{D+2}{D-1}}$};
	\node (E20) at (2,0) {$\ldots$};
	\node (E10) at (1,0) {$\kk[t]^{\binom{D+2}{2}}$};
	\node (E00) at (-.2,0) {$\kk[t]^{\binom{D+2}{1}}$};
 	\node (EN0) at (-.9,-.2) {};
	\node (E60) at (8.2,-.2) {};
 	\node (E05) at (-.9,4.5) {};
 	
     \draw (E05.west) -- (EN0.south west) -- (E60.south);
  \end{tikzpicture}
\end{center}
\end{corollary}

The $E^2$ page can be inferred from this.

\begin{corollary}\label{firste2}
The $E^2$ page of the Mayer-Vietoris spectral sequence of $(\Xs,\XCov)$ is given by:
\begin{center}
  \begin{tikzpicture}[xscale=1.3,yscale=1.1]
	\node (E50) at (5,0) {$t^4\kk[t]$};
 	\node (E40) at (4,0) {$0$}; 
 	\node (E31) at (3,1) {$\frac{t^4\kk[t]}{t^6\kk[t]}$}; 
 	\node (E22) at (2,2) {$\ddots $};
	\node (E13) at (1,3) {$\frac{t^{2D-2}\kk[t]}{t^{2D}\kk[t]}$};
 	\node (E04) at (0,4) {$\frac{t^{2D}\kk[t]}{t^{2D+2}\kk[t]}$}; 
 	\node (E30) at (3,0) {};
	\node (E20) at (2,0) {};
	\node (E10) at (1,0) {};
	\node (E00) at (-.2,0) {$\kk[t]$};
 	\node (EN0) at (-.9,-.2) {};
	\node (E60) at (6,-.2) {};
 	\node (E05) at (-.9,4.5) {};
 	
     \draw (E05.west) -- (EN0.south west) -- (E60.south);
  \end{tikzpicture}
\end{center}
\end{corollary}

\begin{proof}
We have already seen that all persistent homology groups $\PH_q(\Xs)$ for $q\neq 0,D$ are trivial. This means that the corresponding antidiagonals on the $E^{\infty}$ page must consist of trivial modules. As there are no nontrivial differentials to and from $E^r_{p,0}$ for $p\neq D$ for $r\geq 2$, these modules stabilize already on $E^2$. Hence, these are all trivial, except for $E^2_{0,0}\cong\PH_0(\Xs)\cong\kk[t]$. The modules $E^2_{p,q}$ for $q>0$ are isomorphic to $E^1_{p,q}$ since $d^1$ is trivial above the bottom row. Finally, $E^2_{D,0}=\ker d^1_{D,0}$. This can be computed explicitly from the generators, or inductively, as follows. We already know most of $E^2$, so we can use this to our advantage. Namely, we know that
\[
\frac{\ker d^1_{0,0}}{\im d^1_{1,0}}\cong\kk[t]
\]
and for $p=1,\ldots,D-1$ we have
\[
\ker d^1_{p,0}\cong\im d^1_{p+1,0}.
\]
From this, using the first isomorphism theorem, we can inductively infer that for $p=0,\ldots,D-2$
\[
\ker d^1_{p,0}\cong\kk[t]^{\sum_{k=0}^{p+1}(-1)^k\binom{D+2}{p+1-k}}
\]
and using the binomial theorem
\[
\ker d^1_{D-1,0}\cong\frac{t^2\kk[t]}{t^4\kk[t]}\oplus\kk[t]^{\sum_{l=2}^{D+2}(-1)^l\binom{D+2}{D+2-l}}\cong\frac{t^2\kk[t]}{t^4\kk[t]}\oplus\kk[t]^{D+1}.
\]
Since
\[
E^1_{D,0}\cong t^2\kk[t]\oplus\kk[t]^{D+1}
\]
and
\[
\frac{t^2\kk[t]}{t^4\kk[t]}\oplus\kk[t]^{D+1}\cong\im d^1_{D,0}\cong\frac{E^1_{D,0}}{\ker d^1_{D,0}}\cong\frac{t^2\kk[t]\oplus\kk[t]^{D+1}}{\ker d^1_{D,0}},
\]
we finally infer that $\ker d^1_{D,0}\cong t^4\kk[t]$ and thus conclude the proof.
\end{proof}

The $E^\infty$ page can be inferred in a similar fashion.

\begin{corollary}
The $E^\infty$ page of the Mayer-Vietoris spectral sequence of $(\Xs,\XCov)$ is given by:
\begin{center}
  \begin{tikzpicture}[xscale=1.3,yscale=1.1]
	\node (E50) at (5,0) {$t^{2D+2}\kk[t]$};
 	\node (E40) at (4,0) {}; 
 	\node (E30) at (3,0) {};
	\node (E20) at (2,0) {};
	\node (E10) at (1,0) {};
	\node (E00) at (-.2,0) {$\kk[t]$};
 	\node (EN0) at (-.9,-.2) {};
	\node (E60) at (6,-.2) {};
 	\node (E01) at (-.9,0.5) {};
 	
     \draw (E01.west) -- (EN0.south west) -- (E60.south);
  \end{tikzpicture}
\end{center}
\end{corollary}

\begin{proof}
Note that the only nontrivial differential on the $r$-th page, $2\leq r\leq D$, is $d^r_{D,0}$. Note that $E^{R}_{D-r,r-1}$ has already stabilized for $R>r$, as there are no more nontrivial differentials to and from this module. Since $\PH_{D-1}(\Xs)=0$, we can infer that $E^{r+1}_{D-r,r-1}=0$ and that $d^r_{D,0}$ is surjective. A simple inductive argument shows that $E^r_{D,0}\cong t^{2r}\kk[t]$. The spectral sequence collapses at $r=D+1$ where $E^r_{D,0}\cong t^{2D+2}\kk[t]$.
\end{proof}

From these considerations it follows that this example has the following properties:
\begin{itemize}
\item $E^2_{*,0}\itl{\eta}\PH_*(\Nrv)$ holds for $\eta=2$ but not for $\eta<2$,
\item $E^2_{*,0}\itl{\eta}E^\infty_{*,0}$ holds for $\eta=2(D-1)$ but not for $\eta<2(D-1)$,
\end{itemize}
therefore, it attains the bounds from Corollary \ref{secondpage} and Theorem \ref{infinitypage}, so these bounds are in fact sharp.

\subsection{Second Example}

Our second example shows that the bound in Proposition \ref{linking} is also sharp. Let $D\geq1$ and let $\Xs$ be the simplicial complex with vertex set $[0,D+3]$ consisting of all simplices $\sigma\subseteq[0,D+3]$ such that $[1,D+1]\not\subseteq\sigma$ and $\{D+2,D+3\}\not\subseteq\sigma$.

We may visualize $\Xs$ geometrically as a bipyramid consisting of two $(D+1)$-simplices, each of which is subdivided into $D+1$ smaller $(D+1)$-simplices. More specifically, consider the subdivision of the $D$-simplex $[1,D+1]$ into $D+1$ smaller $D$-simplices obtained by adding the point $0$ at the barycenter and connecting it to the vertices (note that this is {\em not} the barycentric subdivision). Then $\Xs$ can be understood as the union of two cones over this subdivision, whose apices are $D+2$ and $D+3$.
%

A cover $\XCov$ of $\Xs$ is given by the cone with apex $D+3$ and the $D+1$ small $(D+1)$-simplices the cone with apex $D+2$ is subdivided into. Specifically, the $0$-th cover element $\XCv_0$ is the full subcomplex of $\Xs$ spanned by $[0,D+1]\cup\{D+3\}$ and for each $i\in[1,D+1]$, the $i$-th cover element $\XCv_i$ is defined as the full subcomplex of $\Xs$ spanned by $[0,D+2]\setminus\{i\}$. In the geometric interpretation mentioned above, the intersection $\XCv_I$ with $0\notin I$ is the cone with apex $D+2$ over the corresponding $(D+1-|I|)$-simplex occurring in the subdivision of the base $D$-simplex $[1,D+1]$, and $\XCv_{I\cup\{0\}}$ is this base $(D+1-|I|)$-simplex.

Next, we define a filtration. The idea is to start with the boundary of the bipyramid $\Xs$ and fill in the $\XCv_i$ one at a time. Let $A$ be the subcomplex of $\Xs$ obtained by removing all simplices $\sigma\subseteq[0,D+3]$ such that $0\in\sigma$. Geometrically, $A$ corresponds to the boundary of the bipyramid $\Xs$. A filtration $\Xs^{-2D}\leq\Xs^0\leq\Xs^2\leq\Xs^4\leq\ldots\leq\Xs^{2D+2}$ of $\Xs$ is defined by
\[
\Xs^{-2D}=A\qquad\text{and}\qquad\Xs^{2j}=A\cup \XCv_0\cup\ldots\cup \XCv_j\qquad\text{for $j\geq 0$}.
\]
We claim that with this filtration, $\Xs$ achieves the relevant bound of $2(D+1)$. To see this, we compute the $E^1$ page of the spectral sequence directly. This corresponds to computing the persistent homology of the $|I|$-fold intersections $\XCv_I$, equipped with the naturally induced filtrations $\XCv_I^{2j}=\XCv_I\cap \Xs^{2j}$.

First, we compute homology of the nerve of the cover of $\Xs$.

\begin{proposition}
The persistent homology of the nerve of $\XCov$ is given by
\[
\PH_q(\Nrv)=\begin{cases}t^{-2D}\kk[t];&q=0,\\
\frac{t^{-2D}\kk[t]}{\kk[t]};&q=D,\\
0;&\text{otherwise.}
\end{cases}
\]
\end{proposition}

\begin{proof}
First recall the definition of a cover element $U_I$. For $I\subseteq [0,D+2]$, $U_I$ consists of all simplices in the space spanned by $[0,D+2]\setminus I $. The nerve becomes non-empty at time $-2D$, since at this time all $D$-simplices $I\subseteq[0,D+1], |I|=D+1,$ are born. This is because for the simplex  $I=[0,D+1]\setminus\{i\}$ in the nerve, where $i>0$,  by definition $U_I^{-2D}$ contains the point $i$. 
On the other hand, if $I=[1,D+1]$, the point $D+2$ is contained in $U_I^{-2D}$. However, the top simplex $[0,D+1]$ only contains the point $0$, which is born at time $0$. It follows from all this that $\Nrv^j$ is a $(D+1)$-simplex for $j\in[0,\infty)$ and the boundary of this $(D+1)$-simplex for $j\in[-2D,0)$.
\end{proof}

Next, we compute the persistent homology of the union.

\begin{proposition}
The persistent homology of $\Xs$ is given by
\[
\PH_q(\Xs)=\begin{cases}t^{-2D}\kk[t];&q=0,\\
\frac{t^{-2D}\kk[t]}{t^{2D+2}\kk[t]};&q=D,\\
0;&\text{otherwise.}
\end{cases}
\]
\end{proposition}

\begin{proof}
For each $j\in[0,D]$, there is a collapse of $\Xs^{2j}$ to $\Xs^{-2D}$. Observing that $\Xs^{2D+2}=\Xs$ is a bipyramid and $\Xs^{-2D}=A$ is its boundary completes the proof.
\end{proof}

In order to compute $E^1$, we describe $\XCv_I^{2j}$ in more familiar terms.

\begin{proposition}
Suppose that $\emptyset\neq I\subseteq[1,D+1]$ and let $j\in[0,D+1]$. Then the following hold:
\begin{itemize}
\item if $j\geq\min I$, $\XCv_I^{2j}$ is a $(D+2-|I|)$-simplex,
\item if $j<\min I$, $\XCv_I^{2j}$ is the join of the boundary $j$-sphere of a $(j+1)$-simplex and a $(D-|I|-j)$-simplex,
\item $\XCv_I^{-2D}$ is a $(D+1-|I|)$-simplex,
\item $\XCv_{I\cup\{0\}}^{2j}$ is a $(D+1-|I|)$-simplex and $\XCv_{I\cup\{0\}}^{-2D}$ is a $(D-|I|)$-simplex,
\item $\XCv_{0}^{2j}$ is a subdivided $(D+1)$-simplex and $\XCv_0^{-2D}$ is the cone over the boundary of a $D$-simplex.
\end{itemize}
In the second and fourth bullet points, we allow $D-|I|-j=-1$ resp. $D-|I|=-1$ and interpret ``$(-1)$-simplex'' as the empty set.
\end{proposition}

\begin{proof}
We begin by proving the first three bullet points. Observe that
\[
\XCv_I^{2j}=\XCv_I\cap\Xs^{2j}=\XCv_I\cap((A\cup\XCv_0)\cup\XCv_1\cup\ldots\cup\XCv_j)=(\XCv_I\cap(A\cup\XCv_0))\cup\XCv_{I\cup\{1\}}\cup\ldots\cup\XCv_{I\cup\{j\}}.
\]
For $j\geq\min I$, one of the terms is $\XCv_{I\cup\{\min I\}}=\XCv_I$, so $\XCv_I^{2j}=\XCv_I$ is the $(D+2-|I|)$-simplex with vertices $[0,D+2]\setminus I$. (Intersecting with $\XCv_i,i>0$, corresponds to removing the vertex $i$.)

For $j<\min I$, $\XCv_{I\cup\{k\}}$ (where $k\in[1,j]$) is the $(D+1-|I|)$-simplex with vertices $[0,D+2]\setminus (I\cup\{k\})$. The first term, $\XCv_I\cap (A\cup\XCv_0)$, consists of two $(D+1-|I|)$-simplices, $[0,D+1]\setminus I$ and $[1,D+2]\setminus I$. So, $\XCv_I^{2j}$ is the complex spanned by all the simplices of the form $J\cup([j+1,D+1]\setminus I)$ where $J$ is a $(j+1)$-element subset of $[0,j]\cup\{D+2\}$. But this means precisely that $\XCv_I^{2j}$ is the simplicial join of the $(D-|I|-j)$-simplex $[j+1,D+1]\setminus I$ and the $j$-sphere, realized as the boundary of the $(j+1)$-simplex $[0,j]\cup\{D+2\}$.

By definition, $\XCv_I^{-2D}=\XCv_I\cap A$ is the $(D+1-|I|)$-simplex spanned by $[0,D+2]\setminus(I\cup\{0\})$.

To prove the fourth bullet point, note that
\[
\XCv_{I\cup\{0\}}^{2j}=\XCv_{I\cup\{0\}}\cap\Xs^{2j}=\XCv_{I\cup\{0\}}\cap((A\cup\XCv_0)\cup\XCv_1\cup\ldots\cup\XCv_j)=\XCv_{I\cup\{0\}}\cup\XCv_{I\cup\{0,1\}}\cup\ldots\cup\XCv_{I\cup\{0,j\}}=\XCv_{I\cup\{0\}}
\]
is the $(D+1-|I|)$-simplex spanned by $[0,D+1]\setminus I$ and $\XCv_{I\cup\{0\}}^{-2D}=\XCv_{I\cup\{0\}}\cap A$ is the $(D-|I|)$-simplex spanned by $[1,D+1]\setminus I$.

The last bullet point follows by definition of $\XCv_0$, namely $\XCv_0^{2j}=\XCv_0\cap\Xs^{2j}=\XCv_0$ is the half of the bipyramid with apex $D+3$, so it is a subdivided $(D+1)$-simplex, and $\XCv_0^{-2D}=\XCv_0\cap\Xs^{-2D}=\XCv_0\cap A$ is the cone with apex $D+3$ over the boundary $(D-1)$-sphere of the base $D$-simplex of the bipyramid.
\end{proof}

Using this fact, we can compute the persistent homology of the intersections $\XCv_I$.

\begin{proposition}
Suppose that $\emptyset\neq I\subseteq[1,D+1]$. Then
\[
\PH_q(\XCv_I)=\begin{cases}\frac{t^{2q}\kk[t]}{t^{2q+2}\kk[t]};&q=D+1-|I|>0,I=[q+1,D+1],\\
t^{-2D}\kk[t]\oplus\frac{\kk[t]}{t^2\kk[t]};&q=D+1-|I|=0,\\
t^{-2D}\kk[t];&q=0<D+1-|I|,\\
0;&\text{otherwise.}
\end{cases}
\]
and for any $I\subseteq[1,D+1]$ we have
\[
\PH_q(\XCv_{I\cup\{0\}})=\begin{cases}t^{-2D}\kk[t];&q=0\text{ and }I\neq[1,D+1],\\
\kk[t];&q=0\text{ and }I=[1,D+1],\\
0;&\text{otherwise.}
\end{cases}
\]
\end{proposition}

\begin{proof}
The join of a sphere and a non-empty simplex is contractible, so $\XCv_I^{2q}$ can only be non-acyclic if it is the join of a sphere and an empty simplex. The previous proposition shows that this occurs precisely if $D=|I|+q-1$ and $I=[q+1,D+1]$ (the latter is required so that $\min I>q$) in which case $\XCv_I^{2q}$ is a $q$-sphere. If $q>0$, this means that $\Hg_q(\XCv_I^{2q})\cong\kk$ and $\Hg_0(\XCv_I^{2q})\cong\kk$ and if $q=0$, it means that $\Hg_q(\XCv_I^{2q})\cong\kk^2.$ In all other cases, including $q=-D$, $\XCv_I^{2q}$ is contractible, so $\Hg_0(\XCv_I^{2q})\cong\kk.$ All other homology groups of $\XCv_I^{2q}$ are $0$.

The second part holds because, once born, $\XCv_{I\cup\{0\}}^{2q}$ is contractible, so we have $\Hg_0(\XCv_{I\cup\{0\}}^{2q})=\kk$ for $q\geq-D$ if $I\neq[1,D+1]$ and for $q\geq0$ otherwise.
\end{proof}

Again, persistent homology has been computed slice-wise, so the remark from the first example applies.

\begin{corollary}
The cover $\XCov$ is $1$-acyclic.
\end{corollary}

As in the previous example, this immediately yields the $E^1$ page.

\begin{corollary}
The $E^1$ page of the Mayer-Vietoris spectral sequence of $(\Xs,\XCov)$ is given by:
\begin{center}
  \begin{tikzpicture}[xscale=1.3,yscale=1.1]
	\node (E50) at (9.2,0) {$\kk[t]^{\binom{D+2}{D+2}}$};
 	\node (E40) at (7.1,0) {$\frac{\kk[t]}{t^2\kk[t]}\oplus(t^{-2D}\kk[t])^{\binom{D+2}{D+1}}$}; 
 	\node (E31) at (4.6,1) {$\frac{t^2\kk[t]}{t^4\kk[t]}$}; 
 	\node (E22) at (3.3,2) {$\ddots $};
	\node (E13) at (2,3) {$\frac{t^{2D-2}\kk[t]}{t^{2D}\kk[t]}$};
 	\node (E04) at (0,4) {$\frac{t^{2D}\kk[t]}{t^{2D+2}\kk[t]}$}; 
 	\node (E30) at (4.6,0) {$(t^{-2D}\kk[t])^{\binom{D+2}{D}}$};
	\node (E20) at (3.3,0) {$\ldots$};
	\node (E10) at (2,0) {$(t^{-2D}\kk[t])^{\binom{D+2}{2}}$};
	\node (E00) at (0,0) {$(t^{-2D}\kk[t])^{\binom{D+2}{1}}$};
 	\node (EN0) at (-.9,-.2) {};
	\node (E60) at (10,-.2) {};
 	\node (E05) at (-.9,4.5) {};
 	
     \draw (E05.west) -- (EN0.south west) -- (E50.south);
  \end{tikzpicture}
\end{center}
\end{corollary}

This allows us to infer the $E^2$ page.

\begin{corollary}
The $E^2=E^\infty$ page of the Mayer-Vietoris spectral sequence of $(\Xs,\XCov)$ is given by:
\begin{center}
  \begin{tikzpicture}[xscale=1.3,yscale=1.1]
 	\node (E40) at (4,0) {$\frac{t^{-2D}\kk[t]}{t^2\kk[t]}$}; 
 	\node (E31) at (3,1) {$\frac{t^2\kk[t]}{t^4\kk[t]}$}; 
 	\node (E22) at (2,2) {$\ddots $};
	\node (E13) at (1,3) {$\frac{t^{2D-2}\kk[t]}{t^{2D}\kk[t]}$};
 	\node (E04) at (0,4) {$\frac{t^{2D}\kk[t]}{t^{2D+2}\kk[t]}$}; 
 	\node (E30) at (3,0) {};
	\node (E20) at (2,0) {};
	\node (E10) at (1,0) {};
	\node (E00) at (0,0) {$t^{-2D}\kk[t]$};
 	\node (EN0) at (-.6,-.2) {};
	\node (E50) at (5,-.2) {};
 	\node (E05) at (-.6,4.5) {};
 	
     \draw (E05.west) -- (EN0.south west) -- (E50.south);
  \end{tikzpicture}
\end{center}
\end{corollary}

\begin{proof}
We claim that the map $d^1_{D+1,0}:E^1_{D+1,0}\to E^1_{D,0}$ is injective. To see this, note that $([0],[0,D+1])$ is a generator of $E^1_{D+1}$ (see Section \ref{sec:specsec} for notation) and its image
\[
d^1([0],[0,D+1])=\sum_{l=0}^{D+1}(-1)^l([0],[0,D+1]\setminus\{l\})
\]
generates a free submodule of $E^1_{D,0}$, since $([0],[0,D+1]\setminus\{l\})$ generates a free submodule of $\PH_0(U_{[0,D+1]\setminus\{l\}})$. This implies that $E^2_{D+1,0}=0$, so there are no nontrivial differentials to or from any $E^r_{p,q}$ for $r\geq 2$. If $q>0$, this is true also for $r=1$. Therefore, the spectral sequence collapses on $E^2$ and $E^1_{p,q}\cong E^2_{p,q}$ for $q>0$. As we have seen, the persistent homology groups $\PH_q(\Xs)$ for $q\neq 0,D$ are trivial. This means that the corresponding antidiagonals on the $E^2=E^{\infty}$ page must consist of trivial modules. Furthermore, $E^2_{0,0}\cong\PH_0(\Xs)\cong t^{-2D}\kk[t]$. Finally, we shall compute $E^2_{D,0}$ explicitly. We already know $\im d^1_{D+1,0}$, so it remains to compute $\ker d^1_{D,0}$ and the corresponding quotient. Using a similar inductive argument as in the proof of Corollary \ref{firste2}, we have
\[
\frac{\kk[t]}{t^2\kk[t]}\oplus(t^{-2D}\kk[t])^{D+2}\cong E^1_{D,0}\cong\im d^1_{D,0}\oplus\ker d^1_{D,0}\cong\ker d^1_{D-1,0}\oplus\ker d^1_{D,0}\cong(t^{-2D}\kk[t])^{D+1}\oplus\ker d^1_{D,0}.
\]
Since these modules are finitely generated, we may conclude that
\[
\ker d^1_{D,0}\cong\frac{\kk[t]}{t^2\kk[t]}\oplus t^{-2D}\kk[t].
\]
In fact, the generators may be deduced from the explicit description of the intersections of the cover elements. Namely, they are given by
\[
a:=(t^{2D}[D+2]-[0],[1,D+1])\qquad\text{and}\qquad b:=([D+2],[1,D+1])+\sum_{l=1}^{D+1}(-1)^l([l],[0,D+1]\setminus\{l\}),
\]
subject to the single relation $t^2 a = 0$. Note that since $([0],[0,D+1]\setminus\{l\})=t^{2D}([l],[0,D+1]\setminus\{l\})$, the generator of $\im d^1_{D+1,0}$ may be written as $t^{2D}b-a$. Therefore, letting $x$ and $y$ be the generators of $\kk[t]\oplus t^{-2D}\kk[t]$, the quotient may be computed as follows:
\[
E^2_{D,0}\cong\frac{\langle x,y\rangle}{\langle t^2 x,t^{2D}y-x\rangle}=\frac{\langle t^{2D}y-x,y\rangle}{\langle t^{2D+2} y,t^{2D}y-x\rangle}\cong\frac{\langle y\rangle}{\langle t^{2D+2} y\rangle}\cong\frac{t^{-2D}\kk[t]}{t^2\kk[t]}.
\]
\end{proof}

The modules $\PH_{D}(\Xs)=\frac{t^{-2D}\kk[t]}{t^{2D+2}\kk[t]}$ and $E^{\infty}_{D,0}=\frac{t^{-2D}\kk[t]}{t^2\kk[t]}$ are $\eta$-interleaved for $\eta=2D$, but not so for any $\eta<2D$. Therefore, this example attains the bound of Proposition \ref{linking}, so this bound is also sharp.

\section{Discussion}
Our initial motivation for this work was algorithmic - given a filtered simplicial complex, it would be computationally desirable to construct a coarser simplicial complex via a cover such that the persistent homology was preserved. This has been done for metric spaces~\cite{sheehy2013linear} but not for more general filtrations. An alternate spectral sequence approach is used for computation of persistence but it does not allow for passing to a coarser representation.  Our results suggest a natural approximation algorithm, where a coarse cover is constructed and the condition of $\eps$-acyclicity is checked locally for each finite intersection. Conversely, the maximum $\eps$ overall finite non-empty intersections could provide the bound. We would then have an explicit error bound relating the persistent homology of the input simplicial complex and the coarser (and presumably smaller) nerve. 

Beyond the initial motivation, our setting of $\kk[t]$-modules and simplicial complexes may seem restrictive. However, these were chosen to make the constructions as explicit as possible and to avoid technical complications. We believe the bounds hold in much greater generality. For example, a natural direction is to consider a sheaf of $q$-tame persistence modules and to use the Leray spectral sequence, of which the Mayer-Vietoris spectral sequence is a special case. We believe the error analysis goes through identically and plan to address this in a separate note. The main technical obstacles are in setting up the spectral sequence so that the differentials are well-defined.

Likewise, the restriction to simplicial complexes is mainly to avoid complications and should hold for CW-complexes or perhaps even suitably nice singular spaces. In general, our results should simplify proving approximation results. It does not require individual sublevel sets of a function to have a good cover at any particular level. In particular, this removes the need to consider the image of a pair of covers. Finally, the $\eps$-acyclicity is a local condition, making it easier to verify in a number of applications.   

Finally we note, this work can also be extended to multidimensional persistence modules. The weaker bound using only interleaving applies directly. The tighter bound does not however, as in our proof that left and right interleavings do not interact (Proposition~\ref{doesnotaddup}), the last case uses the fact that persistence modules have projective dimension of 1, which does not hold for multidimensional persistence modules. 


\section*{Acknowledgments}
The authors would like to thank Don Sheehy for introducing them to the problem and to Petar Pave\v{s}i\'c for suggesting the proof of Proposition~\ref{prop:proof}.



	
	
	


\bibliographystyle{plain}
\bibliography{nerve}

\appendix

\section{Convergence of the Mayer-Vietoris spectral sequence}

The aim of this Appendix is to briefly describe the basic idea of the proof of Theorem \ref{convergence}. Let $(M,\partial^0,\partial^1)$ be the double complex associated to a filtered cover of a filtered simplicial complex, i.e. a pair $(\Xs,\XCov)$, and $(E^r,d^r)$ the spectral sequence associated to this double complex, as defined in Section \ref{sec:specsec}. As mentioned there, the spectral sequence associated to a double complex $(M,\partial^0,\partial^1)$ is just a tool to compute the homology of the associated total complex $(\Tot(M),D)$, namely $(E^r,d^r)$ will converge to $\PH_*(\Tot(M))$. This standard fact can be established by a series of elementary but tedious computations, so we do not replicate the proof here, see for instance \cite{rotman2008introduction}. To prove Theorem \ref{convergence}, it is therefore sufficient to show that the homology of the total complex $\Tot(M)$ is isomorphic to the homology of $(\Xs,f)$.

In fact, the double complex $(M,\partial^0,\partial^1)$ has a geometric counterpart, namely, the {\em filtered Mayer-Vietoris blowup complex $\Xs_\XCov$} associated to $(\Xs,\XCov)$. The total complex $\Tot(M)$ arises as the chain complex associated to $\Xs_\XCov$ and $M$ arises from a filtration on $\Tot(M)$ which in turn is induced by a natural filtration of $\Xs_\XCov$.

So far, we have been working mostly with filtered simplicial complexes, however, in the case at hand, it is slightly more convenient to work with filtered CW complexes and cellular homology. Any filtered simplicial complex $\Xs$ gives rise to a filtered CW complex $\Xs^\CW$ whose cellular homology is isomorphic to the simplicial homology of $\Xs$. In fact, the corresponding chain complexes are isomorphic. Each simplex $\sigma$ in $\Xs$ is assigned a cell $e_{\sigma}$ in $\Xs^\CW$. The cartesian product $\Xs_1^\CW\times\Xs_2^\CW$ of two such complexes again has the structure of a CW complex whose cells are given as $e_{\sigma_1}\times e_{\sigma_2}$ for each pair of simplices $\sigma_1$ in $\Xs_1$ and $\sigma_2$ in $\Xs_2$. The blowup complex is a subcomplex of such a product.

\begin{definition}
The {\em filtered Mayer-Vietoris blowup complex} associated to $(\Xs,\XCov)$ is the filtered CW complex $(\Xs_\XCov,\Filt_\XCov)$, where $\Xs_\XCov\leq\Xs\times\Nrv$ is given by
\[
\Xs_\XCov=\bigcup_{\sigma\in\XCv_I}e_\sigma\times e_I
\]
and the filtration $\Filt_\XCov=(\Xs_\XCov^j)_{j\in\ZZ}$ is given by
\[
\Xs_\XCov^j=\bigcup_{\sigma\in\XCv_I^j}e_\sigma\times e_I.
\]
\end{definition}

Let $(\PC^{\Xs}_*,\partial^{\Xs})$ be the (persistent) cellular chain complex associated to $\Xs^\CW$ and let $(\PC^{\Nrv}_*,\partial^{\Nrv})$ be the cellular chain complex associated to $\Nrv^\CW$. Let $(\PC_*,\partial)$ be the cellular chain complex associated to the blowup complex $\Xs_\XCov$. Explicitly, each $\PC_n$ is the free $\kk[t]$-module, generated by the graded set of all cells $e_\sigma\times e_I$ with $\dim\sigma+\dim I=n$, where the grading is given by $\deg(e_\sigma\times e_I)$, i.e. the birth times of the cells in the filtration of the blowup complex. Since the blowup complex is a subcomplex of $\Xs^\CW\times\Nrv^\CW$, the boundary homomorphisms $\partial_n$ are simply restrictions of the boundary homomorphisms of the chain complex associated to this product. These satisfy the following relation:
\[
\partial_n(e_\sigma\times e_I)=\partial^0_n(e_\sigma)\times e_I+(-1)^{\dim\sigma}e_\sigma\times\partial^1_n(e_I).
\]
Taking into account the isomorphisms between $(\PC^{\Xs}_*,\partial^{\Xs})$ and $(\PC^{\Nrv}_*,\partial^{\Nrv})$ and the corresponding simplicial chain complexes, it follows that $(\PC_*,\partial)\cong(\Tot_*(M),D)$. For comparison, here is the boundary formula for the latter chain complex, written out in full. If $(\sigma,I)$ is a pair with $\dim\sigma=q$ and $\dim I=p$ such that $p+q=n$, we have
\[
D_n(\sigma,I)=\sum_{k=0}^q(-1)^k t^{\deg(\sigma,I)-\deg(\sigma_k,I)}(\sigma_k,I)+(-1)^q\sum_{l=0}^p(-1)^l t^{\deg(\sigma,I)-\deg(\sigma,I_l)} (\sigma,I_l).
\]
This also explains the grading from which the double complex structure of $M$ arises. Namely for each pair $p,q$ with $p+q=n$, let $N_{p,q}\leq\PC_n$ be the submodule freely generated by all cells $e_\sigma\times e_I$ with $\dim\sigma=q$ and $\dim I=p$. Then, we have $\PC_n=\bigoplus_{p+q=n}N_{p,q}$ and $\partial^{\Xs}\times\id$ and $\id\times\partial^{\Nrv}$ respect this grading. The aforementioned isomorphism of $(\PC_*,\partial)\cong(\Tot_*(M),D)$ isomorphically maps the double complex structure of $N$ into that of $M$. Therefore, the homology of the total complex $(\Tot(M),D)$ is precisely the (persistent) cellular homology of the blowup complex. In other words, we have:
\begin{proposition}
The homology of the total complex is isomorphic to the homology of the filtered blowup complex:
\[
\PH_*(\Tot(M),D)\cong\PH^\CW_*(\Xs_\XCov,\Filt_\XCov).
\]
\end{proposition}

It only remains to check that $\PH^\CW_*(\Xs_\XCov,\Filt_\XCov)\cong\PH_*(\Xs,\Filt)$. To see this, it suffices to construct a homotopy equivalence of the two spaces, in the filtered sense. Let $\pi:\Xs_\XCov\to\Xs^\CW$ be the natural projection (to the first component) and let $\pi^j:\Xs_\XCov^j\to(\Xs^j)^\CW$ be the appropriate restriction. It is a standard fact \cite[Proposition 4G.2]{Hatcher} that these projections are homotopy equivalences. Finally, these maps obviously also respect the filtration, i.e. for each $j_1\leq j_2$, the diagram
\begin{center}
\begin{tikzpicture}[scale=1]
\node (B1) at (-1,1) {$\Xs_\XCov^{j_1}$};
\node (B2) at (1,1) {$\Xs_\XCov^{j_2}$};
\node (X1) at (-1,-1) {$(\Xs^{j_1})^\CW$};
\node (X2) at (1,-1) {$(\Xs^{j_2})^\CW$};
\path[->]
(B1) edge (B2)
(X1) edge (X2)
(B1) edge node[left]{$\pi^{j_1}$} (X1)
(B2) edge node[right]{$\pi^{j_2}$} (X2);
\end{tikzpicture}
\end{center}
commutes, because the projections simply forget the second component, whereas the information from the first component remains unchanged. Therefore, as claimed in Theorem \ref{convergence}, the Mayer-Vietoris spectral sequence of $(\Xs,\XCov)$ converges to the cellular persistent homology of $\Xs^\CW$ and therefore to the simplicial persistent homology of $\Xs$.
\end{document}